\documentclass[12pt]{amsart}

\usepackage{amsmath}
\usepackage{latexsym}
\usepackage{amssymb}

\usepackage{amsfonts}

\usepackage{graphics,color,floatrow}

\renewcommand{\proof}{\par\noindent{\it Proof.\ \ }}
\def\qed{\ifmmode\square\else\nolinebreak\hfill
$\Box$\fi\par\vskip12pt}

\def\ov{\overline} 
\def\l{\langle} \def\r{\rangle}

 \def\ZZ{{\mathbb Z}}

\def\calD{{\mathcal D}}

\def\calD{{\mathcal D}}

\def\ZZ{{\mathbb Z}}

\def\K{{\bf K}}

\def\Aut{{\rm Aut}}  
  \def\Ker{{\rm Ker}}

\def\Z{{\bf Z}}

\def\Ga{{\it \Gamma}}

 \def\s{\sigma}

\def\B{{\rm B}} \def\W{{\rm W}}

\newtheorem{theorem}{Theorem}[section]
\newtheorem{lemma}[theorem]{Lemma}%
\newtheorem{corollary}[theorem]{Corollary}%
\newtheorem{proposition}[theorem]{Proposition}%
\newtheorem{hypothesis}[theorem]{Hypothesis}%

\newcommand \be{\begin{equation}}
\newcommand \ee{\end{equation}}
\newcommand \bes{\begin{eqnarray}}
\newcommand \ees{\end{eqnarray}}
\newcommand \bess{\begin{eqnarray*}}
\newcommand \eess{\end{eqnarray*}}

\begin{document}

\title[Complete regular dessins and bicyclic \(2\)-groups]{Complete regular dessins and non-metacyclic bicyclic \(2\)-groups with cyclic commutator subgroup}
\thanks{This work was partially supported  by NSFC: 12061092;
Yunnan Applied Basic Research Projects: {202501AT070016}}

\author{Wenwen Fan}
\address{School of Mathematics\\
Yunnan Normal University\\
Kunming, Yunnan 650500\\
P. R. China}

\address{Yunnan Key Laboratory of Modern Analytical Mathematics and Applications\\  Kunming, Yunnan 650500\\ P. R. China}
\email{fww0871@163.com}

\author{Yuan Shi}
\address{School of Mathematics\\
Yunnan Normal University\\
Kunming, Yunnan 650500\\
P. R. China}
\email{2837017654@qq.com}

\author{Yun Tian Wu}
\address{School of Mathematics\\
Yunnan Normal University\\
Kunming, Yunnan 650500\\
P. R. China}
\email{wyt2027@163.com}

\date\today

\begin{abstract}

We study complete regular dessins with underlying graph \(\K_{2^e,2^f}\), where \(e\geq f\geq 2\),
whose automorphism groups are non-metacyclic.
Together with the cases of coprime orders and odd-prime-power orders, this is one of the basic classes needed for Jones's open
classification problem for regular dessins of complete bipartite graphs. Within the \(2\)-power case,
the present paper treats the subcase in which the commutator subgroup is cyclic; it is the first of three papers
devoted to the non-metacyclic case.

Our approach is group-theoretic. The automorphism group $G$ of a regular dessin of $\K_{2^e,2^f}$ admits an exact bicyclic factorization
$G=\langle b\rangle\langle w\rangle$ of type $(2^e,2^f)$ with $\langle b\rangle\cap\langle w\rangle=1$.
For a fixed group $G$, two exact generating pairs give isomorphic dessins exactly when they lie in the same orbit under $\Aut(G)$.
Thus the classification reduces to three group-theoretic tasks: determining the isomorphism classes of the possible groups $G$,
characterizing their exact bicyclic generating pairs, and computing $|\Aut(G)|$.
Berkovich and Janko divided the relevant non-metacyclic groups, through parameterized presentations, according to whether $G'$ is cyclic,
contains a cyclic subgroup of index $2$, or contains a cyclic subgroup of index $2^r$ with $r>1$.
However, their classification does not identify which parameter values in these
presentations define isomorphic groups; this problem was raised by Janko
in~\cite{J}. We carry out these steps for the cyclic-commutator case. This settles
Janko's isomorphism-classification problem in this case and yields a complete
classification and enumeration of the corresponding regular dessins.

\end{abstract}

\maketitle

\section{Introduction}

One of the central problems in topological graph theory is to classify all symmetric embeddings in orientable surfaces of a given class of graphs.
In this paper, we focus on such embeddings where the underlying graph is a complete bipartite graph.

Through Grothendieck's theory of {\it dessins d'enfants}, these embeddings establish a fundamental correspondence between Riemann surfaces, algebraic curves and Galois groups.
This theory provides a powerful combinatorial approach to the study of Bely\u{\i} surfaces and, in particular, to the absolute Galois group
\(\operatorname{Gal}(\overline{\mathbb Q}/\mathbb{Q})\)
 via its faithful action on some symmetrical bipartite maps, see~\cite{G-J,Grothendieck,Jones-Korea}.
 Consequently, the study of symmetrical orientable bipartite maps has become an active research topic in both algebraic and topological graph theory,
 with a particular focus on the highly symmetric case of complete bipartite graphs.

We use the standard terminology of dessins to represent orientable bipartite maps.
A {\it dessin} $\calD$ is a 2-cell embedding of a connected bipartite graph $\Gamma=(\B\cup \W, E)$ with black and white vertices, into an orientable closed surface $\mathbb{S}$,
where the graph $\Gamma$ is called the {\it underlying graph} and $\mathbb{S}$ the {\it supporting surface.}
If the underlying graph is the complete bipartite graph $\K_{m,n}$, then $\mathcal{D}$ is called a {\it complete dessin of type $(m,n)$}.
Denote by $\Aut\calD$  the automorphism group of $\calD$ which preserves the orientation of the supporting surface and the colors of the vertices.
Then $\Aut\calD$ always acts semi-regularly on the edges $E$ of $\calD$. If this action is transitive (and hence regular), $\mathcal{D}$ is called a {\it regular dessin.}
A regular dessin is termed {\it symmetric} if it admits an external symmetry that transposes the two vertex colors; such dessins correspond to orientably arc-regular bipartite maps.
If a symmetric dessin also admits an orientation-reversing automorphism, it is called {\it reflexible}, representing the highest level of map symmetry; see~\cite{Hu-3,siran-how}.

Regular dessins not only possess a high degree of symmetry, but also correspond to algebraic curves which are natural generalizations of Fermat curves.
They are therefore important in both algebraic geometry and topological graph theory.
The symmetric dessins with underlying graph $\K_{n,n}$ have been completely
classified in a series of papers~\cite{D,DJKNS-1,DJKNS-2,Jones-2010,Jones-2007,JNS,KK,NSZ}.
A natural and more challenging next step, proposed by Gareth Jones in 2007, is
the classification of all regular dessins with underlying graph \(\K_{m,n}\).
Following the classification strategy for symmetric dessins, we consider the
following three basic families of complete regular dessins:
\begin{enumerate}
    \item $\gcd(m,n) = 1$;
    \item $(m,n) = (p^e, p^f)$ with $p$ an odd prime;
    \item $(m,n) = (2^e, 2^f)$.
\end{enumerate}

The first two families have been classified and enumerated in~\cite{Fan-Li-Qu-2018,Fan-2023,Hu K},
and the metacyclic case of the third family has also been classified and enumerated
in~\cite{Fan-2022,Cai-Hu}. Thus the main remaining problem for the family
$(m,n)=(2^e,2^f)$ is the case in which the automorphism group is non-metacyclic.

To formulate this problem group-theoretically, let $\calD$ be a regular dessin of $\K_{2^e,2^f}$ with non-metacyclic
automorphism group and put $G=\Aut(\calD)$. By the dessin--group correspondence
recalled in Section~2, isomorphism classes of such dessins correspond bijectively
to equivalence classes of exact bicyclic triples $(G,b,w)$ satisfying
$G=\langle b\rangle\langle w\rangle$ and $\langle b\rangle\cap\langle w\rangle=1$,
with $|b|=2^e$ and $|w|=2^f$. For a fixed group $G$, two exact generating pairs
give isomorphic dessins precisely when they lie in the same orbit under
$\Aut(G)$. Consequently, the classification problem reduces to the following two
group-theoretic problems:
\begin{enumerate}
\item determine the isomorphism classes of groups in the relevant bicyclic
\(2\)-group families and identify those admitting exact factorizations of the
prescribed orders;
\item characterize the exact bicyclic generating pairs and describe the action
of the corresponding automorphism groups on these pairs.
\end{enumerate}

The first problem has a clear structural starting point. Berkovich and Janko
classified, up to structural families, the \(2\)-groups having exactly one
non-metacyclic maximal subgroup~\cite{B-J,J}; the non-metacyclic bicyclic
\(2\)-groups are precisely the two-generated groups in this class. Thus the
structural families of the relevant groups are known. Their classification does
not, however, determine which groups admit exact factorizations of the required
orders, nor which parameter values within these families define isomorphic
groups. The latter isomorphism problem was raised by Janko in~\cite{J}.

The second problem concerns automorphisms. The automorphism groups of these
\(2\)-groups admit no uniform explicit description across the relevant families,
and even determining their orders is difficult in general. Accordingly, the
central group-theoretic task is to determine, for the relevant groups, their
isomorphism classes, exact bicyclic generating pairs, and automorphism-group
orders. Together with the orbit correspondence above, these data yield the
classification and enumeration of the corresponding regular dessins.

By the Berkovich--Janko classification, a non-metacyclic bicyclic \(2\)-group
belongs to one of the following three structural cases according to its
commutator subgroup \(G'\):
\begin{enumerate}
\item[(i)] \(G'\) is cyclic;
\item[(ii)] \(G'\) contains a cyclic subgroup of index \(2\);
\item[(iii)] \(G'\) contains a cyclic subgroup of index \(2^r\) with \(r>1\).
\end{enumerate}
The present paper treats only case (i). The other two cases are treated
in~\cite{Fan-2,Fan-3}, respectively.

Parts of this three-paper project originated in the
master's theses~\cite{Shi-2026,Su-2025,Yang-2025}, two completed in 2025 and one
in May 2026. An earlier stage of the project was presented by the first author at
the 2025 Workshop on Algebraic Graph Theory and Combinatorics at the Southern
University of Science and Technology in November 2025. All results presented here
have nevertheless been independently rechecked, and complete proofs are given
below.

In a closely related recent preprint, K. Hu~\cite{Hu-2026} studies the
non-metacyclic case via marked exact triples. Our approach is complementary and goes beyond
enumeration: we classify the underlying abstract groups up to isomorphism
(Theorem~\ref{thm:group-classification}), determine the orders of their
automorphism groups (Theorem~\ref{thm:aut-Geta}), and obtain an orbit-theoretic
classification of the corresponding dessins. In particular, this settles
Janko's isomorphism-classification problem for the cyclic-commutator case.

To state our main results conveniently, we introduce a simple notation for the second family of groups in Lemma~\ref{lem:cyclic-commutator}.
For $m,n>1$, $\eta,\delta\in\{0,1\}$ and
$t\in\mathbb Z$, put $z=v^{2^{n-1}}$ and let $G(\eta,t,\delta)$ be the group presented by
\begin{align*}
G(\eta,t,\delta)
={}&\l v,a,x\mid
v^{2^n}=1,\quad v^a=v^{-1+4t},\quad v^x=v^{-1},\\
&[a,x]=v,\quad [v,a^{2^{m-1}}]=1,\quad
a^{2^m}=z^{\eta},\quad x^2=z^{\delta}\r.
\end{align*}
Put $N=\l v,x\r$ and $Z=\l a\r$. Then $G(\eta,t,\delta)=NZ$, where
$N$ is dihedral or generalized quaternion, $Z$ is cyclic, and $Z\cap N\le\l z\r=\Z(N)$.
Moreover, $G'=\l v\r\cong\ZZ_{2^n}$ and $\lvert G(\eta,t,\delta)\rvert=2^{m+n+1}$.
Thus $m,n$ determine the order of $G$ and the cyclic group $G'$, while $\eta,t,\delta$ record the remaining
defining relations within this family.

\vskip0.1in
The main results of this paper are as follows.
\begin{theorem}\label{main-1}
Let $\mathcal{D}$ be a regular dessin whose underlying graph is the complete bipartite graph $\K_{2^e,2^f}$, where $e\geq f\geq 2$.
Assume that the automorphism group $G = \Aut(\mathcal{D})$ is non-metacyclic and has a cyclic commutator subgroup.
Then one of the following holds.
\begin{itemize}
\item[(1)] $G\cong G_0$, where
\[
G_0=\l x,y\mid x^4=y^2=z^2=1,\ [x,y]=z,\ [x,z]=[y,z]=1\r,
\]
and in this case $e=f=2$.
\item[(2)] $G\cong G(\eta,t,\delta)$ for some integers $m,n>1$.
In this case \((e,f)=(n+1,m)\) and, necessarily, \(n\geq m-1\).
Then, up to isomorphism, $G$ is exactly one of the following groups:
\begin{itemize}
\item[(i)] $G(0,2^{\lambda},0)$ for every $\lambda\in\mathcal L_{m,n}$;
\item[(ii)] $G(0,2^{\lambda},1)$ for every $\lambda\in\mathcal L_{m,n}$ with $\lambda\geq n-m$;
\item[(iii)] $G(1,2^{\lambda},0)$ when $n=m$ and $\lambda\in\mathcal L_{n,n}$,
\end{itemize}
where
\[
\mathcal L_{m,n}
=\bigl\{\lambda\in\mathbb Z:
\max\{0,n-m-1\}\le\lambda\le n-2\bigr\}.
\]
All groups listed in (i)--(iii) are pairwise non-isomorphic and each of them occurs.
Equivalently, these are precisely the isomorphism classes in
Theorem~\ref{thm:group-classification} that admit exact bicyclic generating
pairs of type \((2^e,2^f)\).
\end{itemize}
\end{theorem}

\medskip
\noindent\textit{Remark on Theorem~\ref{main-1}.}
The choices of representatives and the omissions in Theorem~\ref{main-1} are
justified as follows:
\begin{enumerate}
\item[(1)] Here \(\lambda\) denotes the normalized exponent of the power of
\(2\) in \(t\). If \(t\not\equiv0\pmod{2^{n-2}}\), then \(\lambda\) is the
exponent of \(2\) in \(t\). If \(t\equiv0\pmod{2^{n-2}}\), then
\(v^a=v^{-1}\); in this case we take the representative \(t=2^{n-2}\) and set
\(\lambda=n-2\). Consequently, no value \(\lambda>n-2\) is needed.
\item[(2)] When \(\lambda=n-m-1\), we have
\(G(0,2^{\lambda},1)\cong G(0,2^{\lambda},0)\).
\item[(3)] When \(\eta=\delta=1\), we have
\(G(1,2^{\lambda},1)\cong G(1,2^{\lambda},0)\) for \(\lambda\geq n-m\),
while \(G(1,2^{\lambda},1)\cong G(0,2^{\lambda},0)\) for
\(\lambda=n-m-1\).
\end{enumerate}

Having classified the possible automorphism groups in Theorem~\ref{main-1}, we now enumerate the corresponding dessins.

\begin{theorem}\label{main-2}
For $e\geq f\geq 2$, let
$n_G(\mathcal{K}_{2^e,2^f})$ denote the number of non-isomorphic regular dessins with underlying graph
$\K_{2^e,2^f}$ whose automorphism group is isomorphic to $G$, where $G$ is one of the
groups in Theorem~\ref{main-1}. Then the following hold.
\begin{itemize}
\item[(1)] If $G=G_0$, then $e=f=2$ and
\[
n_G(\mathcal{K}_{2^e,2^f})=1.
\]
\item[(2)] If $G=G(\eta,t,\delta)$, then
\[
n_G(\mathcal{K}_{2^e,2^f})=
\begin{cases}
2^{2f-3} & \text{if }e=f\geq3,\\
2^{2f-2} & \text{if }e>f.
\end{cases}
\]
\end{itemize}
\end{theorem}

\vskip0.2in
\section{Preliminaries}

In this section, we present several key results that are essential for the subsequent proofs.\\

\vskip0.2in
\subsection{Methods and ideas}~~~~~~~~~~~~~~~~~
\vskip0.15in
Let $\calD$ be a regular dessin with underlying graph $\Ga=(B\cup W, E)=\K_{m,n}$,
and let $G=\Aut\calD$ be the automorphism group which preserves the orientation of the supporting surface and the colors of the vertices.
Our method uses the following result, which is described in detail in~\cite{CFLZ, Fan-2022}.

\begin{proposition}\label{fundamental}
If $G$ is a finite group, then the following are equivalent:
\begin{itemize}
\item[(1)]$G\cong\Aut\calD$ for some regular dessin $\calD$ with underlying graph $\K_{m,n}$;
\item[(2)]$G$ has cyclic subgroups $\l b\r\cong\ZZ_m$ and $\l w\r\cong\ZZ_n$ such that $$G=\l b\r\l w\r~\mbox{and}~\l b\r\cap \l w\r=1.$$
\end{itemize}
\end{proposition}

In this situation, we say that the group $G$ is an {\it exact bicyclic group}, namely, it is a product of two cyclic subgroups with trivial intersection.
The pair $(b,w)$ is called a {\it generating pair}, and the ordered pair $(o(b),o(w))$ is called its {\it type}.

For convenience, the regular dessin corresponding to the triple $(G,b,w)$ is denoted by $$\calD(G,b,w).$$
By~\cite{Fan-2022}, the dessin $\calD(G,b,w)$ defined as in Proposition~\ref{fundamental} has $m+n$ vertices,
$mn$ edges, and face length $2|bw|$. Hence its genus is
\[{\bf g}=1+\frac{1}{2}(|E|-|V|-|F|)=1+\frac{|G|}{2}(1-\frac{1}{|b|}-\frac{1}{|w|}-\frac{1}{|bw|}).\]

The following criterion determines when two regular dessins are isomorphic; see~\cite[Lemma 2.8]{CFLZ} for details.

\begin{lemma}\label{iso-criterion}
Two regular dessins $\calD(G,b,w)$ and $\calD(G,b',w')$ are isomorphic if and only if there is an automorphism $\s\in \Aut(G)$ such that $(b,w)^\s=(b',w')$.
\end{lemma}

If $\calD(G,b,w)\cong\calD(G,b',w')$, the two triples $(G,b,w)$ and $(G,b',w')$ are called {\it equivalent.}
This result implies that for each triple $(G,b,w)$, if $G=\langle b\rangle\langle w\rangle\cong\ZZ_m\ZZ_n$ and $\l b\r\cap \l w\r=1$,
then the triple $(G,b,w)$ gives rise to a dessin $\calD(G,b,w)$ with underlying graph $\K_{m,n}$ on which $G=\Aut\calD$ acts edge-regularly.
Therefore, by Proposition~\ref{fundamental} and Lemma~\ref{iso-criterion}, classifying and enumerating regular dessins of $\K_{m,n}$
is equivalent to finding all non-equivalent triples $(G,b,w)$ under the action of $\Aut(G)$.

Let $$\Omega=\{(b,w)\in G\times G~|~\l b\r\l w\r=G, |b|=m,|w|=n, \l b\r\cap \l w\r=1\}.$$
According to the above analysis, the non-isomorphic regular dessins of $\K_{m,n}$ associated with $G$ correspond to the orbits of $\Aut(G)$ on $\Omega$.

\begin {lemma}\label{semi-regular}
$\Aut(G)$ acts on $\Omega$ semi-regularly.
\end {lemma}

\proof Let $\sigma$ be an arbitrary automorphism in $\Aut G$ such that $(b,w)^\sigma=(b,w)$, where $(b,w)\in \Omega$.
Then $b^\sigma=b$ and $w^\sigma=w$. It follows that $\sigma$ fixes every generating element of \(G\), so \(\sigma=1\).
Therefore, $\Aut(G)$ acts semi-regularly on $\Omega$.\qed
This lemma implies that the number of regular dessins of $\K_{m,n}$ associated with the fixed group $G$
is equal to $\frac{|\Omega|}{|\Aut(G)|}$.\\

\vskip0.2in
\subsection{Non-metacyclic bicyclic $2$-groups with cyclic commutator subgroup}~~~~~~~~~~~~~~
\vskip0.15in

This paper, together with the subsequent two articles~\cite{Fan-2,Fan-3}, is devoted to completing the classification of regular dessins
 whose underlying graph is $\Gamma=\K_{2^e,2^f}$ and whose automorphism group is non-metacyclic, where $e\geq f\geq2$.
The approach involves groups $G$ that are non-metacyclic bicyclic $2$-groups.
Here $d(G)$ denotes the minimum number of generators of $G$. According to~\cite{B-J}, a non-metacyclic $2$-group is bicyclic
if and only if $d(G)=2$ and it has exactly one maximal subgroup that is non-metacyclic.

\begin{proposition}\label{pro-1}
Let $G$ be a non-metacyclic $2$-group. Then $G$ is bicyclic if and only if $d(G)=2$ and $G$ has exactly one non-metacyclic maximal subgroup.
\end{proposition}

For a group $G$, let $G'$ and $\Phi(G)$ denote its commutator subgroup and Frattini subgroup, respectively, and for $x, y \in G$, let $[x, y] = x^{-1}y^{-1}xy$ denote the commutator.

Berkovich and Janko gave a structural classification of all non-metacyclic $2$-generated $2$-groups
that have exactly one non-metacyclic maximal subgroup (see Chapter 87 of~\cite{B-J}).
Thus the possible presentation types of non-metacyclic bicyclic $2$-groups are known.
This presentation-based classification does not by itself determine which presentations in a parameterized family define isomorphic groups.

In the Berkovich--Janko classification, the commutator subgroup $G'$, which is nontrivial and abelian,
falls into the following three cases:
\begin{itemize}
\item $G'$ is cyclic;
\item $G'$ contains a cyclic subgroup of index $2$;
\item $G'$ contains a cyclic subgroup of index $2^r$ with $r>1$.
\end{itemize}
Since the present paper is concerned only with the first case, we record only the corresponding structural
classification in~\cite[Theorems 87.9, 87.12]{B-J}; the full lists for the other two cases will be given
in the subsequent papers in this series.

\begin{lemma}\label{lem:cyclic-commutator}
Let $G$ be a $2$-group with exactly one non-metacyclic maximal subgroup and $d(G)=2$. If $G'$ is cyclic, then one of the following holds.
\begin{itemize}
\item[(1)] If $G'\cong\ZZ_2$, then $G$ has the presentation
\[
G=\l x,y\mid x^{2^n}=y^2=z^2=1,\ n\geq 2,\ [x,y]=z,\ [x,z]=[y,z]=1\r.
\]
This group is non-metacyclic and minimal non-abelian, with $G/G'$ of type $(2^n,2)$.
Moreover, $G'=\l z\r$ is a minimal cyclic subgroup of $G$, and $G=\l x\r\l xy\r$.
\item[(2)] If $G'\cong\ZZ_{2^n}$ with $n>1$, then $G=NZ$ for some $m>1$, where $N$ is a normal subgroup of $G$ of the form
\[
N=\l v,x\mid v^{2^n}=1,\ z=v^{2^{n-1}},\ x^2\in\l z\r,\ v^x=v^{-1}\r,
\]
which is dihedral or generalized quaternion, and $Z=\l a\r$ is cyclic such that
$Z\cap N\leq\l z\r=\Z(N)$, $\lvert Z/(Z\cap N)\rvert=2^m$, $[a,x]=v$,
$v^a=v^{-1+4t}$ for some $t\in\ZZ$, and $[v,a^{2^{m-1}}]=1$.
Here $G'=\l v\r\cong\ZZ_{2^n}$, $\Phi(G)=G'\l a^2\r$,
and $G=\l ax\r\l a\r$ is a product of two cyclic subgroups.\\
\end{itemize}
\end{lemma}

\vskip0.2in
\subsection{A basic property of the automorphism group}~~~~~~~~~~~~~~~~~~~~~~~~~~~~~~~~~
\vskip0.15in
In this subsection we record a basic property of the automorphism group that will be used repeatedly
in the subsequent classification.
Let $\Gamma=\K_{2^e,2^f}$ be a complete bipartite graph and  let $\calD$ be a regular dessin of $\Gamma$ with $G=\Aut\calD$ being non-metacyclic.
Then by Proposition~\ref{fundamental}, $G$ is a non-metacyclic exact bicyclic $2$-group of type $(2^e,2^f)$
satisfying the following hypothesis.

\begin{hypothesis}\label{hypo}
$G=\l b\r\l w\r$ is non-metacyclic, where $o(b)=2^e$ and $o(w)=2^f$, $e\geq f\geq 2$
such that $\l b\r\cap\l w\r=1$.
\end{hypothesis}

The following lemma shows that the {\it exponent} of $G$, namely the highest order of an element of $G$, is $2^e$.
Its proof is given in~\cite[Lemma~2.7]{Fan-2022}.

\begin{lemma}\label{exponent}
The exponent of a group \(G\) satisfying Hypothesis~\ref{hypo} is \(2^e\).
\end{lemma}

For the case $G'$ cyclic considered in this paper, Lemma~\ref{lem:cyclic-commutator} yields two possible
families of groups. The first is treated in Section~3, where its structure under an exact bicyclic factorization
is determined. The second is the parameterized family $G(\eta,t,\delta)$, depending on $m,n,\eta,t,\delta$.
Its isomorphism classes, exact bicyclic factorizations, and automorphism-group
orders are treated in Sections~4.1, 4.2, and 4.3, respectively.

\vskip0.2in

\section{The case $G'\cong\ZZ_2$}

By Lemma~\ref{lem:cyclic-commutator}, the automorphism group in this case has the presentation
\[
G=\l x,y\mid x^{2^n}=y^2=z^2=1,\ [x,y]=z,\ [x,z]=[y,z]=1\r
\]
for some $n\geq 2$. We shall show that the existence of an exact bicyclic factorization of type $(2^e,2^f)$
forces $n=2$, so that $G\cong G_0$, where $G_0$ is as stated in Theorem~\ref{main-1}(1). For the present family,
\(|G|=2^{n+2}\); moreover, $G'=\l z\r$ is a minimal cyclic subgroup of $G$, and $G=\l x\r\l xy\r$.
Thus every element of $G$ can be written as $x^i(xy)^j$ for some integers $i,j$.

Let $$\Omega=\{(b,w)=(x^i(xy)^j,x^k(xy)^{\ell})|o(b)=2^e, o(w)=2^f, \l b,w\r=G, \l b\r\cap\l w\r=1\}.$$
Then, by Proposition~\ref{fundamental}, each generating pair $(b,w)\in \Omega$ gives rise to a regular dessin $\calD(G,b,w)$.
Moreover, by Lemma~\ref{iso-criterion} and Lemma~\ref{semi-regular}, the number of non-isomorphic regular dessins of $\K_{2^e,2^f}$ associated with $G $ is equal to $\frac{|\Omega|}{|\Aut(G)|}$.
We therefore need to determine both $|\Omega|$ and $|\Aut(G)|$.\\

\vskip0.2in
\subsection{Exact bicyclic generating pairs}~~~~~~~~~~~~~
\vskip0.15in
Since every element of \(G\) is of the form \(x^i(xy)^j\), every pair in \(\Omega\)
has the form \((x^i(xy)^j,x^k(xy)^\ell)\), and determining \(|\Omega|\) is equivalent to determining all
admissible quadruples \((i,j,k,\ell)\). To this end, we first recall a criterion for two elements to generate
a \(2\)-generated \(2\)-group.
\begin{lemma}\label{lem:generating-matrix}
Let $G=\l b, w\r=\l x,y\r$ be a $2$-generated $2$-group. If
$b=y^ix^j$ and $w=y^kx^\ell$ for some integers \(i,j,k,\ell\), then
$$\left|
  \begin{array}{cc}
    i& j  \\
    k & \ell  \\
    \end{array}
\right|\not\equiv 0\ (\mbox{mod}~2).  \eqno(3.1)$$
\end{lemma}
\proof
Let $\overline{G}=G/\Phi(G)$, and for $g\in G$, denote by $\overline{g}$ its image in $\overline{G}$.
Then
$$\left\{\begin{array}{ll}
\overline{b}=\overline{y}^i\overline{x}^j \\
\overline{w}=\overline{y}^k\overline{x}^\ell
\end{array} \right. $$
Since $\l \ov b, \ov w\r=\l \ov x,\ov y\r=\ov G\cong\ZZ_2\times\ZZ_2$,
we may view $\ov G$ as a $2$-dimensional vector space over $\mathbb{F}_2$.
Then both $\{\ov b,\ov w\}$ and $\{\ov x,\ov y\}$ are bases of this space.
The transition matrix between them is
$$\left[
  \begin{array}{cc}
    i& j  \\
    k & \ell  \\
    \end{array}
\right],$$
which must be non-singular. Hence its determinant is non-zero modulo $2$, giving the desired conclusion. \qed

We now apply this criterion to \(G\) and determine the admissible values of $i,j,k,\ell$.

\begin{lemma}\label{lem:G0-generating-pairs}
Let $G$, as above, be the automorphism group of a regular dessin whose underlying graph is $\K_{2^e,2^f}$,
where $n\geq2$ and $e\geq f\geq 2$.
Then $n=2$ and
\[
\Omega=\bigl\{(x^i(xy)^j,x^k(xy)^{\ell}):
1\leq i,j,k,\ell\leq4,\ k\equiv j\pmod{2},\ i\equiv\ell\equiv j+1\pmod{2}\bigr\},
\]
which has cardinality $|\Omega|=32$.
\end{lemma}

\proof Let $\calD$ be a regular dessin with underlying graph $\K_{2^e,2^f}$ and automorphism group $G$ as above.
By Proposition~\ref{fundamental},  $G=\l b\r\l w\r\cong \ZZ_{2^e}\ZZ_{2^f}$ is a non-metacyclic exact bicyclic $2$-group, where $e\geq f\geq 2$.
Then $2^{e+f}=|G|=2^{n+2}$,  so $e+f=n+2$.  According to Lemma~\ref{exponent}, we have $n\leq e$, and thus $f\leq 2$.
This implies that $$(e,f)=(n,2).$$

Now \(G=\langle x\rangle\langle xy\rangle\), where \(xy\) has order \(2^n\) because
 $$(xy)^{2^{n-1}}=(x^2z)^{2^{n-2}}=x^{2^{n-1}}z^{2^{n-2}}\neq 1~\mbox{and}~(xy)^{2^n}=(x^{2^{n-1}}z^{2^{n-2}})^2=1.$$
Let $(b,w)=(x^i(xy)^j,x^k(xy)^{\ell})$ be an exact bicyclic generating pair of type $(2^n, 2^2)$,
where $i,j,k,\ell$ are integers.
Applying Lemma~\ref{lem:generating-matrix} with generators $x$ and $xy$, we see that $i,j,k,\ell$ satisfy condition (3.1).
Thus $\{i,j\}$ and $\{k,\ell\}$ cannot both be even.
A direct calculation gives
\[
(xy)^j=\begin{cases}
x^jy & \text{if } j\equiv1\pmod4,\\
x^jz & \text{if } j\equiv2\pmod4,\\
x^jyz & \text{if } j\equiv3\pmod4,\\
x^j & \text{if } j\equiv0\pmod4.
\end{cases}\eqno(3.2)
\]
and
\[
(x^i(xy)^j)^2=\begin{cases}
x^{2(i+j)}z^{i+j} & \text{if } j \text{ is odd},\\
x^{2(i+j)} & \text{if } j \text{ is even}.
\end{cases}\eqno(3.3)
\]

Iterating, we obtain
\[
1\ne b^{2^{n-1}}=(x^i(xy)^j)^{2^{n-1}}
=\begin{cases}
x^{2^{n-1}(i+j)}z^{2^{n-2}(i+j)} & \text{if } j \text{ is odd},\\
x^{2^{n-1}(i+j)} & \text{if } j \text{ is even}.
\end{cases}\eqno(3.4)
\]

\[
1\ne w^2=(x^k(xy)^\ell)^2=\begin{cases}
x^{2(k+\ell)}z^{k+\ell} & \text{if } \ell \text{ is odd},\\
x^{2(k+\ell)} & \text{if } \ell \text{ is even}.
\end{cases}\eqno(3.5)
\]
and $$1= w^{2^2}=(x^k(xy)^{\ell})^{2^2}=x^{2^2(k+\ell)},$$
which implies that $i+j$ is odd and $2^{n-2}||(k+\ell)$.

Suppose that $n>2$. Then by Eqs.~(3.4)--(3.5), we have
$$b^{2^{n-1}}=x^{2^{n-1}}=w^2,$$ which contradicts the condition $\l b\r\cap\l w\r=1$. Thus, $n=e=f=2$.
In this case, $G$ has order $16$, and its elements are represented uniquely as $x^i(xy)^j$ with $1\leq i,j\leq4$.
Moreover, $i+j$ and $k+\ell$ are both odd,
and so the set of exact bicyclic generating pairs of type $(2^2,2^2)$ is
\[
\Omega=\bigl\{(x^i(xy)^j,x^k(xy)^{\ell}):
1\leq i,j,k,\ell\leq4,\ k\equiv j\pmod{2},\ i\equiv\ell\equiv j+1\pmod{2}\bigr\}.
\]
This set has cardinality $|\Omega|=32$.\qed

\vskip0.2in
\subsection{The automorphisms of \(G_0\)}~~~~~~~~~~~~~~~
\vskip0.15in
According to Lemma~\ref{lem:G0-generating-pairs}, if $G$ is the automorphism group of a regular dessin
in the present case, then $G\cong G_0$, where \(G_0\) is as stated in Theorem~\ref{main-1}(1).
To complete the classification, we determine the order of $\Aut(G_0)$.

\begin{lemma}\label{lem:aut-G0}
Let $G_0=\l x,y~|~x^{2^2}=y^2=z^2=1,[x,y]=z, [x,z]=[y,z]=1 \r$. Then $$|\Aut(G_0)|=32.$$
\end{lemma}
\proof Since $G_0=\l x,y\r=\l x\r\l xy\r$, each automorphism $\sigma\in\Aut(G_0)$ can be written in the following form:
\[\sigma:\left\{
\begin{array}{ll}
x\rightarrow x^i(xy)^j           \\
y\rightarrow x^k(xy)^\ell
\end{array}, \right.\]
where the parameters~$1\leq i,j,k,\ell \leq 2^2$ satisfy Eq.~(3.1).

Since $o(\sigma(x))=2^2$ and $o(\sigma(y))=2$, it follows from
Eqs.~(3.3) and (3.4) that $i+j,k,\ell$ are all odd; that is, \(i\) and
\(j\) have opposite parity. Furthermore, by Eq.~(3.2), we have
\((xy)^{\ell}=x^{\ell}yz^m\) and \((xy)^j=x^jy^sz^t\), where
$m,s,t=0$ or $1$. In particular, $s=1$ if $j$ is odd and $s=0$ if $j$ is even.
Then
 \[\sigma(x)\cdot\sigma(y)=x^i(xy)^{j+\ell}(x^k)^{(xy)^{\ell}}=x^i(xy)^{j+\ell}(x^k)^{x^\ell yz^m}=x^i(xy)^{j+\ell}(x^y)^k=x^i(xy)^{j+\ell}x^kz\]
and
\[\sigma(y)\cdot\sigma(x)=x^k(xy)^{j+\ell}(x^i)^{(xy)^{j}}=x^k(xy)^{j+\ell}(x^i)^{x^jy^sz^t}=x^k(xy)^{j+\ell}(x^i)^{y^s}=x^k(xy)^{j+\ell}x^i,\]
Hence $$\sigma(z)=[\sigma(x),\sigma(y)]=(x^k(xy)^{j+\ell}x^i)^{-1}x^i(xy)^{j+\ell}x^kz=z.$$
As $z\in \Z(G)$, the relations $[\sigma(x),\sigma(z)]=[\sigma(y),\sigma(z)]=1$ hold automatically.
Therefore, the order of $\Aut(G_0)$ equals the product of all admissible choices of $i,j,k,\ell$, i.e., $$|\Aut(G_0)|=(2\cdot2+2\cdot 2)\cdot2\cdot2=32.$$\qed

\subsection{Classification and enumeration of regular dessins}~~~~~~~~~
\vskip0.1in
Combining Lemma~\ref{lem:G0-generating-pairs} and Lemma~\ref{lem:aut-G0}, we obtain the following result.
\begin{lemma}\label{lem:count-G0}
The complete bipartite graph $\K_{2^e,2^f}$ has a unique regular dessin with automorphism group $G_0$;
moreover, this dessin has genus $1$.
\end{lemma}

\proof Let $\calD$ be a regular dessin with underlying graph $\K_{2^e,2^f}$ and automorphism group $G_0$.
By Lemma~\ref{lem:G0-generating-pairs} and Lemma~\ref{lem:aut-G0}, we have $n=2$ and $|\Aut(G_0)|=|\Omega|$.
Hence, there exists a unique regular dessin corresponding to $G_0$, which, up to isomorphism, is given by
$$\calD=\calD(G,b,w)=\calD(G,x^i(xy)^j,x^k(xy)^{\ell}),$$ where $j$ and $k$ are odd, $i$ and $\ell$ are both even.
By Eq.~(3.2), we obtain $(xy)^j=x^jyz^s$ and $(xy)^{\ell}=x^{\ell}z^t$, where $s,t=0$ or $1$.
Thus, $b^2=(x^{i+j}yz^s)^2=x^2z,~w^2=(x^{k+\ell}z^t)^2=x^2$ and $$bw=x^i(xy)^j\cdot x^k(xy)^{\ell}=x^{i+j}yz^s\cdot x^{k+\ell}z^t=x^{i+j+k+\ell}yz^{s+t+1}.$$
It follows that $|bw|=2$. Hence the genus of $\calD$ equals
\[{\bf g}=1+\frac{|G|}{2}(1-\frac{1}{|b|}-\frac{1}{|w|}-\frac{1}{|bw|})=1.\]
This completes the proof.\qed

\vskip0.2in
\section{The case $G'\cong\ZZ_{2^n}$ with $n>1$}
\vskip0.1in

In this section we study the regular dessins with underlying graph $\K_{2^e,2^f}$ whose automorphism group
\(G\) satisfies $G'\cong\ZZ_{2^n}$ with \(n>1\). By Lemma~\ref{lem:cyclic-commutator}, $G=NZ$ for some
$m,n>1$, where
\[
N=\l v,x\mid v^{2^n}=1,\ z=v^{2^{n-1}},\ x^2\in\l z\r,\ v^x=v^{-1}\r
\quad\text{and}\quad
Z=\l a\r.
\]
Here $N$ is normal and dihedral or generalized quaternion, and
$Z\cap N\leq\langle z\rangle=\Z(N)$ with $|Z/(Z\cap N)|=2^m$.
The defining relations include $[a,x]=v$, $[v,a^{2^{m-1}}]=1$ and $v^a=v^{-1+4t}$,
where $t\in\mathbb Z$.
Since \(x^2\in\langle z\rangle\), \(a^{2^m}\in Z\cap N\leq\langle z\rangle\), and
\(z\) has order \(2\), each of \(x^2\) and \(a^{2^m}\) is either \(1\) or \(z\).
We therefore define the unique parameters \(\eta,\delta\in\{0,1\}\) by
\[
x^2=z^\delta\quad\text{and}\quad a^{2^m}=z^\eta.
\]
Thus \(\delta\) and \(\eta\) record whether \(x^2\) and \(a^{2^m}\) are equal to \(1\) or \(z\).
This gives precisely the parameterized group \(G(\eta,t,\delta)\) introduced before Theorem~\ref{main-1}.
The order, derived subgroup and Frattini subgroup of \(G\) are given by
\[
|G|=2^{m+n+1},~ G'=\langle v\rangle\cong\ZZ_{2^n}
~\text{ and }~ \Phi(G)=G'\langle a^2\rangle.
\]
Moreover, $G=\langle a\rangle\langle ax\rangle$ is a product of two cyclic subgroups.

Since $v$ has order $2^n$, the relation $v^a=v^{-1+4t}$ depends on
$t$ only modulo $2^{n-2}$. We therefore distinguish the two cases
\[
  t\equiv0\pmod{2^{n-2}}
  \qquad\text{and}\qquad
  t\not\equiv0\pmod{2^{n-2}}.
\]
Let \(\nu_2\) denote the \(2\)-adic valuation, namely the exponent of \(2\) in
the prime factorization of a nonzero integer. We define the normalized exponent
\[
  \lambda_0=
  \begin{cases}
    n-2 & t\equiv0\pmod{2^{n-2}},\\
    \nu_2(t) & t\not\equiv0\pmod{2^{n-2}}.
  \end{cases}
\]
Accordingly, $v^a=v^{-1}$ precisely when \(t\equiv0\pmod{2^{n-2}}\). In this
case we choose the representative \(t=2^{n-2}\), for which
\(\lambda_0=n-2\). Otherwise \(\lambda_0=\nu_2(t)\). The bound
\(\lambda_0\le n-2\) is a choice of representative, not a restriction on the
\(2\)-adic valuation of \(t\).

Thus, with $m,n$ fixed, the group depends on the parameters $\eta,t,\delta$.
We first determine the necessary relations among the parameters.

\begin{lemma}\label{parameter-relation}
Let $\lambda_0$ be the representative exponent defined above.
If $\lambda_0\le n-3$, then
\[\lambda_0+3\leq n \leq m+1+\lambda_0.\]
\end{lemma}
\proof
By hypothesis, \(n\geq\lambda_0+3\). Moreover, since \(\lambda_0\le n-3\),
the representative exponent satisfies \(\lambda_0=\nu_2(t)\). Write
$t=2^{\lambda_0}t_0$ with $t_0$ odd.

Now set \(r=-1+4t\). From $v^a=v^r$ and $[v,a^{2^{m-1}}]=1$, we obtain
\[
v^{r^{2^{m-1}}}=v^{a^{2^{m-1}}}=v.
\]
Since \(v\) has order $2^n$, it follows that
\[
r^{2^{m-1}}\equiv1\pmod{2^n}.
\]
Moreover,
\[
r^2=(-1+4t)^2=1-8t(1-2t).
\]
Because \(t=2^{\lambda_0}t_0\) with \(t_0\) odd and \(1-2t\) is odd, there is
an odd integer \(u\) such that \(r^2=1-2^{\lambda_0+3}u\). A direct binomial
expansion gives, for \(m>1\),
\[
\nu_2\!\left(r^{2^{m-1}}-1\right)
=(m-2)+(3+\lambda_0)
=m+1+\lambda_0.
\]
The congruence above then implies
\[
m+1+\lambda_0\geq n.
\]
Combining the two inequalities gives
\[
\lambda_0+3\leq n\leq m+1+\lambda_0.
\]\qed

\subsection{Isomorphism classes of $G(\eta,t,\delta)$}~~~~~~~~~~~~~~~~~~~~~~~~~~~~~~~
\vskip0.15in

In this subsection we determine the isomorphism classes of the groups
$G(\eta,t,\delta)$ with admissible parameters. This is necessary because different parameter
triples may define the same group, and the later count of generating-pair orbits and dessins
must be performed once for each group isomorphism class.

\begin{lemma}\label{lem:iso-t}
Let \(G=G(\eta,t,\delta)\) with $\eta,\delta\in\{0,1\}$, and let $\lambda_0$ be the
representative exponent defined above. Then
\[
  G(\eta,t,\delta)\cong G(\eta,2^{\lambda_0},\delta).
\]
where $\max\{0,n-m-1\}\le\lambda_0\le n-2.$

\end{lemma}

\proof
Let \(G=G(\eta,t,\delta)\). Since \(v=[a,x]\), the group is generated by \(a,x\), so
$G=\l a,x\r$. Put \(r=4t-1\).
First suppose that \(\lambda_0\le n-3\), so \(n\ge3\) and
\(\lambda_0=\nu_2(t)\). Put \(s=2^{\lambda_0+2}-1\). Both \(r\) and \(s\) are
units congruent to \(-1\pmod4\); in particular, \(r,s\in U_{2^n}\), where
\(U_{2^n}\) denotes the group of units modulo \(2^n\). Since
\(U_{2^n}=\{\pm5^m\}\), write
\(r=-5^\alpha\) and \(s=-5^\beta\). By the binomial theorem,
\(\nu_2(5^m-1)=2+\nu_2(m)\) for \(m\ge1\). Since
\[
\nu_2(r+1)=\nu_2(4t)=\lambda_0+2
\quad\text{and}\quad
\nu_2(s+1)=\nu_2(2^{\lambda_0+2})=\lambda_0+2,
\]
we obtain \(\nu_2(\alpha)=\nu_2(\beta)=\lambda_0\). Writing
\(\alpha=2^{\lambda_0}\alpha_0\) and \(\beta=2^{\lambda_0}\beta_0\)
with \(\alpha_0,\beta_0\) odd, the congruence
\[
\alpha_0q\equiv\beta_0\pmod{2^{n-\lambda_0-2}}
\]
has an odd solution \(q\). Consequently,
\(r^q=(-5^\alpha)^q=-5^{\alpha q}\equiv-5^\beta=s\pmod{2^n}\).
For the remaining case \(\lambda_0=n-2\), the representative is \(t=2^{n-2}\);
then \(r\equiv-1\pmod{2^n}\) and \(s=2^n-1\), so we take \(q=1\).
Choose such an odd \(q\) and put $h=a^q$. Since the order of \(a\) is a power of \(2\),
the odd power \(a^q\) also generates \(\langle a\rangle\). Hence
\[ G(\eta,t,\delta)=\l a,x\r=\l h,x\r. \]
Now we check the relations satisfied by $h$:
\[ h^{2^m}=(a^{2^m})^{q}=a^{2^m}=z^{\eta},~\ov{v}:=[h,x]=[a^{q},x]=v^{(r^q-1)/(r-1)}, \]
where $(r^q-1)/(r-1)$ is odd; hence $[\ov{v},h^{2^{m-1}}]=1$, $z=\ov{v}^{2^{n-1}}$ and $\ov{v}^x=\ov{v}^{-1}$. Then the last relation $v^a=v^r$ gives
\[ \ov{v}^{h}=(v^{(r^q-1)/(r-1)})^{a^{q}}=v^{r^q(r^q-1)/(r-1)}=\ov{v}^{r^q}, \]
which forces $\ov{v}^h=\ov{v}^{2^{\lambda_0+2}-1}$ because $r^q\equiv 2^{\lambda_0+2}-1\pmod{2^n}$. Thus
\begin{align*}
      G(\eta,t,\delta)=\l h,x~|~&[h,x]=\ov{v},~\ov{v}^{2^n}=1,~\ov{v}^{h}=\ov{v}^{2^{\lambda_0+2}-1},\\
      &\ov{v}^x=\ov{v}^{-1},~h^{2^m}=z^{\eta},~x^2=z^{\delta}  \r\cong G(\eta,2^{\lambda_0},\delta).
\end{align*}

It remains to verify the asserted range of \(\lambda_0\).
By its definition, \(0\le\lambda_0\le n-2\). The lower bound follows from
Lemma~\ref{parameter-relation} when \(\lambda_0\le n-3\), and is immediate
for the endpoint \(\lambda_0=n-2\). Hence
\[ \max\{0,n-m-1\}\le\lambda_0\le n-2. \]
This completes the proof.\qed

The preceding lemma shows that the odd factor of \(t\) does not affect the isomorphism class;
the group depends at most on \(\lambda_0\). We now show that distinct values of \(\lambda_0\)
give non-isomorphic groups.
The possible normalized exponents form the set
\[
\mathcal L_{m,n}
=\left\{\lambda\in\mathbb Z:
\max\{0,n-m-1\}\le\lambda\le n-2\right\}.
\]
By Lemma~\ref{lem:iso-t}, for each representative exponent \(\lambda_0\in\mathcal L_{m,n}\),
we may use \(t=2^{\lambda_0}\) as the representative of the corresponding isomorphism class.
In this representative, the defining relation becomes
\[
v^a=v^{-1+4\cdot2^{\lambda_0}}=v^{2^{\lambda_0+2}-1}.
\]

\begin{lemma}\label{center-relation}
Let \(\eta,\delta\in\{0,1\}\), and let \(\lambda_1,\lambda_2\in\mathcal L_{m,n}\) with
\(\lambda_1\ne\lambda_2\). Then
\[
G(\eta,2^{\lambda_1},\delta)\ncong G(\eta,2^{\lambda_2},\delta).
\]
\end{lemma}

\proof
Let \(\lambda_0\in\mathcal L_{m,n}\), and set
\(G=G(\eta,2^{\lambda_0},\delta)\) and \(r=2^{\lambda_0+2}-1\).
Since $G=NZ=\l v,x\r\l a\r$ with $\l v,x\r$ dihedral or quaternion, each element of \(G\)
has the form $a^ix^jv^k$. Put $g=a^ix^jv^k\in\Z(G)$, where $1\le i\le 2^m$,
$j\in\{0,1\}$ and $1\le k\le 2^n$. Then
\[ a^ix^jv^k=g=g^a=a^i(x^jv^k)^a=a^i(xv^{-1})^jv^{rk}. \]
Suppose that $j=1$. Then $v^{(r-1)k-1}=1$, a contradiction because $(r-1)k-1$ is odd. Thus $j=0$, that is, $(r-1)k\equiv 0\pmod{2^n}$. Moreover, $k\equiv 0\pmod{2^{n-1}}$. Then
\[ a^iv^k=g=g^x=(a^iv^k)^x=(av)^iv^{-k}=a^iv^{(r^i-1)/(r-1)-k}, \]
which forces
\[ (r^i-1)/(r-1)\equiv 2k\equiv 0\pmod{2^n}\]
and so \[i\equiv0\pmod{2^{n-\lambda_0-1}}.\]
Hence \(i\) and \(k\) have \(2^{m-(n-\lambda_0-1)}\) and \(2\) choices,
respectively. Therefore,
\[|\Z(G)|=2^{m+\lambda_0-n+2}.\]
This value is strictly increasing in \(\lambda_0\). Therefore, representatives
corresponding to distinct elements of \(\mathcal L_{m,n}\) are non-isomorphic.
This completes the proof.\qed

Consequently, for fixed \((\eta,\delta)\), the isomorphism classes are represented by
the pairwise non-isomorphic groups
\[
G(\eta,2^{\lambda_0},\delta),
\]
where \(\lambda_0\in\mathcal L_{m,n}\).
Fix \(\lambda_0\in\mathcal L_{m,n}\). The remaining ambiguity is in the pair
\((\eta,\delta)\in\{0,1\}^2\), which gives the four groups
\[
  G(0,2^{\lambda_0},0),\quad
  G(0,2^{\lambda_0},1),\quad
  G(1,2^{\lambda_0},0),\quad
  G(1,2^{\lambda_0},1).
\]
We now determine their isomorphism classes.

\begin{lemma}\label{lem:iso-delta}
Let $\max\{0,n-m-1\}\le\lambda_0\le n-2$. The four groups
$G(\eta,2^{\lambda_0},\delta)$, $\eta,\delta\in\{0,1\}$, fall into exactly three
isomorphism classes. They are represented by
\[
G(0,2^{\lambda_0},0),\qquad G(1,2^{\lambda_0},0),\qquad G(\eta,2^{\lambda_0},1),
\]
where the last group is $G(0,2^{\lambda_0},1)$ if $\lambda_0\ge n-m$, and
$G(1,2^{\lambda_0},1)$ if $\lambda_0=n-m-1$.
\end{lemma}

\proof
Let $G(\eta,2^{\lambda_0},\delta)=\langle a,x,v\rangle$ with $r=2^{\lambda_0+2}-1$.

We first use the number of involutions to eliminate all impossible isomorphisms. Put
\[
  I(\eta,\lambda_0,\delta)
  =\{g\in G(\eta,2^{\lambda_0},\delta):g^2=1\}.
\]
Groups with different numbers of involutions are non-isomorphic.

Let \(q=|G'|=2^n\) and \(M=2^{m-1}\). By the power formula, every involution lies in one of the four
cosets
\[
  G',\qquad a^MG',\qquad xG',\qquad a^MxG'.
\]
Indeed, write an arbitrary element as \(g=a^ix^jv^k\). Modulo \(G'\), it is determined by
\(a^i x^j\). In the quotient \(\overline G=G/G'\), we have \(\bar x^2=1\), and
\(\bar a\) and \(\bar x\) commute because \([a,x]=v\in G'\). Moreover,
\(\bar a^{2^m}=1\). Hence \(\bar g^2=1\) forces \(\bar a^{2i}=1\), so
\(i\equiv0\) or \(M\pmod{2^m}\); since \(x^2=z^\delta\in G'\), the exponent \(j\) is only
taken modulo \(2\). These four possibilities for \((i,j)\) give precisely the four displayed cosets.
We now count the involutions in these cosets.

The coset $G'$ contains the unique involution $z$.  In $xG'$, an
element has the form $xw$, and
$(xw)^2=x^2=z^{\delta}$; hence this coset contributes $q$ involutions
when $\delta=0$, and none when $\delta=1$.
An element of the coset $a^MG'$ has the form $a^Mv^u$. It is an involution
precisely when
\[ 1=(a^Mv^u)^2=a^{2^m}v^{2u}=z^{\eta}v^{2u}=v^{2^{n-1}\eta}v^{2u}. \]
It follows that $2u\equiv 2^{n-1}\eta\pmod{2^n}$. Thus the coset
$a^MG'$ provides exactly $2$ involutions.

Finally, let $g=a^Mxv^{\rho}\in a^M xG'$. Put \(w=a^Mx\) and
\(Q=(r^M-1)/(r-1)\). Direct computation gives
\[
  w^2=x^2a^{2^m}v^Q=z^{\eta+\delta}v^Q,
  \qquad
  w^{-1}vw=v^{-r^M}.
\]
The second relation implies \(v^{\rho}w=wv^{-\rho r^M}\), and hence
\[
  g^2=w^2v^{\rho(1-r^M)}
  =z^{\eta+\delta}v^{Q\{1-\rho(r-1)\}}.
\]
The factor $1-\rho(r-1)$ is odd. If $\lambda_0\le n-3$, a binomial expansion
gives \(v_2(Q)=m+\lambda_0\), whereas
Lemma~\ref{parameter-relation} gives \(m+\lambda_0\ge n-1\). If
\(\lambda_0=n-2\), then \(r\equiv-1\pmod{2^n}\), so the even number of
terms in \(Q=\sum_{j=0}^{M-1}r^j\) gives \(Q\equiv0\pmod{2^n}\).
Thus \(v^{Q\{1-\rho(r-1)\}}=z^S\), where
\(S\equiv2^{m+\lambda_0-n+1}\pmod2\). Therefore \(g^2=1\) is equivalent to
\[
  \eta+\delta+2^{m+\lambda_0-n+1}\equiv0\pmod 2.
\]
Thus this coset contributes $q$ involutions precisely when
$\eta=\delta$ if $m+\lambda_0\ge n$, and precisely when
$\delta=\eta+1$ if $m+\lambda_0=n-1$; in all remaining cases it contributes
no involutions.

Combining these contributions, we obtain
{\small
\[
|I(\eta,\lambda_0,\delta)|=
\begin{cases}
3+2q & \text{if } (\eta,\delta)=(0,0)\text{ and }\lambda_0\ge n-m,\\
3    & \text{if } (\eta,\delta)=(0,1)\text{ and }\lambda_0\ge n-m,\\
3+q  & \text{if } (\eta,\delta)=(1,0)\text{ and }\lambda_0\ge n-m,\\
3+q  & \text{if } (\eta,\delta)=(1,1)\text{ and }\lambda_0\ge n-m,\\
3+q  & \text{if } (\eta,\delta)=(0,0)\text{ and }\lambda_0=n-m-1,\\
3+q  & \text{if } (\eta,\delta)=(0,1)\text{ and }\lambda_0=n-m-1,\\
3+2q & \text{if } (\eta,\delta)=(1,0)\text{ and }\lambda_0=n-m-1,\\
3    & \text{if } (\eta,\delta)=(1,1)\text{ and }\lambda_0=n-m-1.
\end{cases}
\]}
Since isomorphic groups have the same number of involutions, groups for which
these values differ are non-isomorphic. Thus only the following pairs still
require consideration:
\begin{itemize}
      \item[(i)] If $\lambda_0=n-m-1$, then possibly $G(0,2^{\lambda_0},0)\cong G(0,2^{\lambda_0},1)$,
      \item[(ii)] If $\lambda_0\ge n-m$, then possibly $G(1,2^{\lambda_0},0)\cong G(1,2^{\lambda_0},1)$.
\end{itemize}
Equality of involution numbers is only a necessary condition, so these two
pairs remain candidates for isomorphism.

We now establish both isomorphisms
explicitly. Put $h=a^Mx$.
Then $G(\eta,2^{\lambda_0},\delta)=\l a,x\r=\l a,h\r$. By the preceding
\(2\)-adic estimates, \(r^M\equiv1\pmod{2^n}\), and hence
\[
  [a,h]=[a,a^Mx]=[a,x]=v,
  \quad\text{and}\quad
  v^h=(v^{a^M})^x=v^{-r^M}=v^{-1}.
\]
Thus all defining relations except the one involving \(h^2\) already have
the required form. Substituting \(\rho=0\) into the preceding computation of
\(g^2\) gives
\[
  h^2=z^{\eta+\delta}v^Q
  =z^{\eta+\delta+2^{m+\lambda_0-n+1}}.
\]
Therefore
\begin{align*}
      G(\eta,2^{\lambda_0},\delta)=\l a,h~|~&[a,h]=v,~v^{2^n}=1,~v^{a}=v^{2^{\lambda_0+2}-1},~v^h=v^{-1},\\
      &a^{2^m}=z^{\eta},~h^2=z^{\delta+\eta+2^{m+\lambda_0-n+1}}  \r\cong G(\eta,2^{\lambda_0},\delta+\eta+2^{m+\lambda_0-n+1}).
\end{align*}
Thus one of the following statements holds:
\begin{itemize}
      \item[(i)] If $\lambda_0=n-m-1$, then $G(0,2^{\lambda_0},0)\cong G(0,2^{\lambda_0},1)$,
      \item[(ii)] If $\lambda_0\ge n-m$, then $G(1,2^{\lambda_0},0)\cong G(1,2^{\lambda_0},1)$.
\end{itemize}

In conclusion, there are exactly $3$ pairwise non-isomorphic groups, that is,
\[ G(0,2^{\lambda_0},0),~G(1,2^{\lambda_0},0),~G(\eta,2^{\lambda_0},1), \]
where either $\eta=0$ if $\lambda_0\ge n-m$, or $\eta=1$ if $\lambda_0=n-m-1$.
This completes the proof.
\qed

Combining Lemmas~\ref{lem:iso-t}--\ref{lem:iso-delta}, we obtain the following
classification of the family \(G(\eta,t,\delta)\).

\begin{theorem}\label{thm:group-classification}
Let \(G=G(\eta,t,\delta)\) be as in Lemma~\ref{lem:cyclic-commutator}(2).
Then the isomorphism class of \(G\) is represented by exactly one of the
following three pairwise non-isomorphic groups:
\[
G(0,2^{\lambda_0},0),\qquad
G(1,2^{\lambda_0},0),\qquad
\begin{cases}
G(0,2^{\lambda_0},1) & \text{if }\lambda_0\ge n-m,\\
G(1,2^{\lambda_0},1) & \text{if }\lambda_0=n-m-1.
\end{cases}
\]
where \(\lambda_0\in\mathcal L_{m,n}\) is the representative exponent of \(G\).
\end{theorem}

\begin{corollary}
For fixed \(m,n>1\), the family \(G(\eta,t,\delta)\) with admissible parameters
has exactly \(3\min\{m,n-1\}\) isomorphism classes.
\end{corollary}
\proof
For each \(\lambda_0\in\mathcal L_{m,n}\), Lemma~\ref{lem:iso-delta} gives exactly
three isomorphism classes. Since
\(|\mathcal L_{m,n}|=\min\{m,n-1\}\), the total number is
\(3\min\{m,n-1\}\). \qed

\vskip0.3in
\subsection{Exact bicyclic generating pairs}~~~~~~~~~~~~~~~~~~~~~~~~~~~~~~~~~~~
\vskip0.15in

As in the preceding section, we first identify the set \(\Omega\) of exact
bicyclic generating pairs of \(G=G(\eta,t,\delta)\) and determine its
cardinality. To determine \(\Omega\), we first establish the exponentiation
formula for an arbitrary generator of \(G\).

For convenience, in this subsection we extend the \(2\)-adic valuation by
setting \(\nu_2(0)=+\infty\), and write \(\lambda=\nu_2(t)\).

\begin{lemma}\label{lem:power-formula}
Let $G = G(\eta,t,\delta)$ be as defined above.
Every element of \(G\) occurring in some generating pair can be written as
\(g=a^i(ax)^j\), where at least one of \(i\) and \(j\) is odd. For \(s\ge2\),
the power formulas for \(g\) are given as follows.
\begin{itemize}
\item[(i)] If \(\lambda\ge n-2\), then
$$g^{2^s} = (a^i(ax)^j)^{2^s} =
\begin{cases}
a^{2^s(i+j)} & \text{if } i \text{ is odd},\\[4pt]
a^{2^s(i+j)} v^{2^{s-1}j} & \text{if } i \text{ is even}.
\end{cases}\eqno(4.1)$$
\item[(ii)] If \(\lambda\le n-3\), then
$$g^{2^s} = (a^i(ax)^j)^{2^s} =
\begin{cases}
a^{2^s(i+j)} v^{2^{s+\lambda+\nu_2(i+j)}u_1}
  & \text{if } j \text{ is odd and } i \text{ is odd},\\[4pt]
a^{2^s(i+j)} v^{2^{s-1} u_2}
  & \text{if } j \text{ is odd and } i \text{ is even},\\[4pt]
a^{2^s(i+j)} v^{2^{s+\lambda+\nu_2(j)}u_3}
  & \text{if } j \text{ is even and } i \text{ is odd}.
\end{cases}\eqno(4.2)$$
where $u_1, u_2, u_3$ are odd integers.
\end{itemize}
\end{lemma}

\proof Let $G=G(\eta,t,\delta)=NZ$ with $m>1$, where
$$N=\langle v,x \mid v^{2^n}=1,\ n>1,\ x^2\in \langle z\rangle,\ z=v^{2^{n-1}},\ v^x=v^{-1}\rangle \lhd G,$$
$Z=\langle a\rangle$, $Z\cap N\leq \langle z\rangle=Z(N)$, and moreover,
$$|Z/(Z\cap N)|=2^m,~ [a,x]=v,~ [v,a^{2^{m-1}}]=1,~\mbox{and}~v^a=v^{-1+4t},$$
where \(t\in\mathbb Z\).
In this case, $G'=\langle v\rangle\cong \mathbb{Z}_{2^n}$ for $n\geq 2$, $\Phi(G)=G'\langle a^2\rangle$, and $G=\langle a\rangle\langle ax\rangle$ is a product of two cyclic subgroups.
Therefore, every generating element $g$ of $G$ can be written as $g=a^i(ax)^j$, where $1\leq i\leq o(a)$, $1\leq j\leq o(ax)$, and at least one of $i$ and $j$ is odd.

Let $r=-1+4t$. Then $v^a=a^{-1}va=v^r$, and a direct calculation yields
$$(av)^\xi=a^\xi v^{r^{\xi-1}+r^{\xi-2}+\cdots+r+1}=a^\xi v^{(r^\xi-1)/(r-1)} \eqno (4.3)$$ for some positive integer $\xi$.
Since $v^x=v^{-1}$ and $[a,x]=v$, we have $a^x=av$, $[a,x^2]=[v,x^2]=1$, and
\begin{align*}
(ax)^2 &= (xav)^2 = x^2a^xv^xa^x = x^2a^2v,\\
(ax)^3 &= (x^2a^2v)(ax) = x^3(a^2)^xv^xa^x = x^3a^3v^{r^2+1},\\
(ax)^4 &= x^3a^3v^{r^2+1}ax = x^4(a^3)^x(v^{r^2+1})^xa^x = x^4a^4v^{r^2+1},\\
(ax)^5 &= x^4a^4v^{r^2+1}ax = x^5a^5v^{r^4+r^2+1},\\
&\;\;\vdots
\end{align*}
By induction, we obtain
$$(ax)^\theta=
\begin{cases}
x^{\theta}a^{\theta}v^{\lceil \theta/2\rceil}
  & \text{if } \lambda\ge n-2,\\
x^\theta a^\theta v^{(r^{2\lceil \theta/2\rceil}-1)/(r^2-1)}
  & \text{if } \lambda\le n-3.
\end{cases}$$
where $\left\lceil \frac{\theta}{2}\right\rceil=\begin{cases}
(\theta+1)/2 & \text{if } \theta \text{ is odd},\\
\theta/2 & \text{if } \theta \text{ is even}.
\end{cases}$
In particular, for $\theta=2^s$ with $s\geq 2$, we have
$$(ax)^{2^s}=\begin{cases}
x^{2^s}a^{2^s}v^{2^{s-1}}
  & \text{if } \lambda\ge n-2,\\
x^{2^s}a^{2^s}v^{(r^{2^s}-1)/(r^2-1)}
  & \text{if } \lambda\le n-3.
\end{cases}$$
When \(\lambda\le n-3\), we have
\(2^{\lambda+2}\parallel(r+1)\). A calculation shows that
\(2^{\lambda+2+s}\parallel(r^{2^s}-1)\) and
\(2^{\lambda+3}\parallel(r^2-1)\);
consequently, $2^{s-1}\parallel (r^{2^s}-1)/(r^2-1)$.
In summary, $$(ax)^{2^s}=x^{2^s}a^{2^s}v^{2^{s-1}u}, \eqno(4.4) $$
where $u$ is an odd integer. We now treat the two cases
\(\lambda\ge n-2\) and \(\lambda\le n-3\) separately.

\medskip
\noindent\textbf{Case 1:} \(\lambda\ge n-2\), equivalently,
\(r\equiv-1\pmod{2^n}\).
\medskip
\noindent In this case, we have
\[
g=a^i(ax)^j=a^i\cdot x^ja^jv^{\lceil j/2\rceil}
=\begin{cases}
x^ja^{i+j}v^{j/2} & \text{if } j \text{ is even and } i \text{ is odd},\\
x^ja^{i+j}v^{(j+1)/2} & \text{if } j \text{ is odd and } i \text{ is even},\\
x^ja^{i+j}v^{(j-1)/2} & \text{if } j \text{ is odd and } i \text{ is odd}.
\end{cases}
\]
A direct computation gives
\[
g^2=\begin{cases}
a^{2(i+j)} & \text{if } j \text{ is even and } i \text{ is odd},\\
x^2a^{2(i+j)}v^j & \text{if } j \text{ is odd and } i \text{ is even},\\
x^2a^{2(i+j)} & \text{if } j \text{ is odd and } i \text{ is odd}.
\end{cases}
\]
Since \(x^4=1\), the factor \(x^2\) in the last two lines becomes \(1\)
after one further squaring. Repeating this calculation gives, for \(s\ge2\),
\[
g^{2^s}=\begin{cases}
a^{2^s(i+j)} & \text{if } i \text{ is odd},\\
a^{2^s(i+j)}v^{2^{s-1}j} & \text{if } i \text{ is even}.
\end{cases}
\]

\medskip
\noindent\textbf{Case 2:} \(\lambda\le n-3\), equivalently,
\(r\not\equiv-1\pmod{2^n}\).
\medskip
\noindent In this case, we have
\[
g=a^i(ax)^j=a^i\cdot x^ja^jv^{(r^{2\lceil j/2\rceil}-1)/(r^2-1)}
=\begin{cases}
x^ja^{i+j}v^{\,r^j\left(\frac{r^i-1}{r-1}\right)+\frac{r^{j+1}-1}{r^2-1}}
  & \text{if } j \text{ is odd},\\
x^ja^{i+j}v^{\,\frac{r^j-1}{r^2-1}}
  & \text{if } j \text{ is even}.
\end{cases}
\]
Squaring once gives
\[
g^2=\begin{cases}
x^2a^{2(i+j)}v^{\frac{(r^{i+j}-1)(r^j+1)}{r^2-1}}
  & \text{if } j \text{ is odd},\\
a^{2(i+j)}v^{\frac{(r^{i+j}+1)(r^j-1)}{r^2-1}}
  & \text{if } j \text{ is even}.
\end{cases}
\]
Since \(x^4=1\), one further squaring yields
\[
g^{2^2}=\begin{cases}
a^{2^2(i+j)}v^{\frac{(r^{2(i+j)}+1)(r^{i+j}-1)(r^j+1)}{r^2-1}}
  & \text{if } j \text{ is odd},\\
a^{2^2(i+j)}v^{\frac{(r^{2(i+j)}+1)(r^{i+j}+1)(r^j-1)}{r^2-1}}
  & \text{if } j \text{ is even}.
\end{cases}
\]
Repeating this calculation gives, for every \(s\ge2\),
\[
g^{2^s}=\begin{cases}
a^{2^s(i+j)}v^{\frac{(r^{i+j}-1)(r^j+1)\prod_{u=1}^{s-1}(r^{2^u(i+j)}+1)}{r^2-1}}
  & \text{if } j \text{ is odd},\\
a^{2^s(i+j)}v^{\frac{(r^{i+j}+1)(r^j-1)\prod_{u=1}^{s-1}(r^{2^u(i+j)}+1)}{r^2-1}}
  & \text{if } j \text{ is even}.
\end{cases}\eqno(4.5)
\]

We now determine the highest power of $2$ dividing the exponent of $v$ in (4.5).
Since $r=-1+4t$, it follows that $r^{2^u}\equiv 1 \pmod{2^{u+2}}$
for each $u=1,2,\ldots,s-1$, and hence
\[r^{2^u(i+j)}\equiv 1 \pmod{2^{u+2}}, ~u=1,2,\ldots,s-1,\]
which implies that $2\parallel (r^{2^u(i+j)}+1)$ for each $u=1,2,\ldots,s-1$. Consequently,
\[2^{s-1}\parallel \prod_{u=1}^{s-1}\bigl(r^{2^u(i+j)}+1\bigr).\]
On the other hand, if $j$ is odd, then
\[
\nu_2(r^j+1)=2+\lambda,
\quad\text{and}\quad
\nu_2(r^{i+j}-1)=
\begin{cases}
2+\lambda+\nu_2(i+j) & \text{if } i \text{ is odd},\\
1 & \text{if } i \text{ is even}.
\end{cases}
\]
Consequently,
\[
\nu_2\left(\frac{(r^{i+j}-1)(r^j+1)}{r^2-1}\right)=
\begin{cases}
1+\lambda+\nu_2(i+j) & \text{if } i \text{ is odd},\\
0 & \text{if } i \text{ is even}.
\end{cases}
\]
Similarly, if $j$ is even, then $i$ must be odd (since at least one of $i$ and $j$ is odd), and
\[
\nu_2\left(\frac{(r^{i+j}+1)(r^j-1)}{r^2-1}\right)
=1+\lambda+\nu_2(j).
\]
Combining the above computations, we obtain
\[g^{2^s}=(a^i(ax)^j)^{2^s}=\begin{cases}
a^{2^s(i+j)} v^{2^{s+\lambda+\nu_2(i+j)}u_1} & \text{if } j \text{ is odd and } i \text{ is odd},\\
a^{2^s(i+j)} v^{2^{s-1} u_2} & \text{if } j \text{ is odd and } i \text{ is even},\\
a^{2^s(i+j)} v^{2^{s+\lambda+\nu_2(j)}u_3} & \text{if } j \text{ is even and } i \text{ is odd}.
\end{cases},\]
where $u_1,u_2,u_3$ are odd integers. This completes the proof. \qed

The following lemma determines the set \(\Omega\) of exact bicyclic generating
pairs of \(G=G(\eta,t,\delta)\).

\begin{lemma}\label{lem:generating-pairs-second}
Let \(G=G(\eta,t,\delta)\) be as in
Lemma~\ref{lem:cyclic-commutator}(2), and let
\[
\Omega=\{(b,w)=(a^i(ax)^j,a^k(ax)^{\ell}):
o(b)=2^e,\ o(w)=2^f,\ \l b\r\l w\r=G,\ \l b\r\cap\l w\r=1\}
\]
be the set of all exact bicyclic generating pairs of type \((2^e,2^f)\) of
\(G\), where \(e\ge f\ge2\). Then
\[
(e,f)=(n+1,m)
\quad\text{with}\quad
m-1\le n\le m+\lambda+1,
\]
where \(\lambda=\nu_2(t)\), and the following statements hold.
\begin{itemize}
\item[(1)] If \(\eta=0\) and $n=m-1$, then $|\Omega|=2^{2m+2n-1}=2^{4m-3}.$
\item[(2)] If \(\eta=0\) and $m\le n\le m+\lambda+1$, then
$|\Omega|=2^{2n+2m-2}.$
\item[(3)] If \(\eta=1\), then $n=m$ and $|\Omega|=2^{2m+2n-2}.$
\end{itemize}
\end{lemma}

\proof Let $(b,w)=(a^i(ax)^j,a^k(ax)^{\ell})\in\Omega$, where
\(1\le i,k\le o(a)\) and \(1\le j,\ell\le o(ax)\).
Then $(b,w)$ is a generating pair of $G$ such that $o(b)=2^e$, $o(w)=2^f$, $\langle b\rangle\langle w\rangle=G$, and $\langle b\rangle\cap\langle w\rangle=1$.
Hence $2^{e+f}=|G|=2^{m+n+1}$, so $e+f=m+n+1$.
It follows from Lemma~\ref{lem:generating-matrix} that the parameters $i,j,k,\ell$ satisfy
$$\left| \begin{array}{cc}
    i& j  \\
    k & \ell  \\
    \end{array}\right|\not\equiv 0\ (\mbox{mod}~2).\eqno(4.6)$$
To determine the size of $\Omega$, we need to clarify the ranges of the parameters $i,j,k,\ell$.

To apply (4.1) and (4.2) below, we need \(e-1\ge2\), i.e. \(e\ge3\).
In fact, we show the stronger inequality \(e\ge n+1\). By
Lemma~\ref{exponent}, \(\exp(G)=2^e\), so it suffices to prove
\(\exp(G)\ge2^{n+1}\). If \(a^{2^n}=1\), then \(x^{2^n}=1\) and (4.4)
give
\[
(ax)^{2^n}=v^{2^{n-1}u}\ne1.
\]
Otherwise, \(o(a)>2^n\). Since \(G\) is a \(2\)-group, either case gives
\(\exp(G)\ge2^{n+1}\). Therefore \(e\ge n+1\ge3\), so (4.1) and (4.2)
apply in all calculations below.

Suppose first that $j$ and $i$ are both odd. Then $i+j$ is even.
According to Eqs.~(4.1) and (4.2), we have
\[b^{2^{e-1}}=(a^i(ax)^j)^{2^{e-1}}=\begin{cases}
a^{2^{e-1}(i+j)}=1 & \text{if } \lambda\ge n-2,\\
a^{2^{e-1}(i+j)}v^{2^{e-1+\lambda+\nu_2(i+j)}u_1}=1
  & \text{if } \lambda\le n-3.
\end{cases},\]
a contradiction. Therefore, $i$ and $j$ have different parity.

Since \(a^{2^m}=z^\eta\) and \(|Z/(Z\cap N)|=2^m\), the two possible
intersections are distinguished by the value of \(\eta\). We treat them separately.

\vskip0.1in
\textbf{Case 1:} \(\eta=0\), equivalently \(N\cap Z=1\).
\vskip0.1in
Since $|Z/(Z\cap N)|=2^m$ with $m>1$, it follows that $Z=\l a\r\cong\ZZ_{2^m}$, and by Lemma~\ref{exponent},  we have $e\geq \max\{m,n\}$ and $f\leq \min\{n+1,m+1\}$.
\vskip0.1in
\textbf{Case 1.1. Suppose first that $i$ is even and $j$ is odd.}
It follows from Eq.~(4.6) that $k$ is odd, and by Eqs.~(4.1)--(4.2) we have
\[1=b^{2^{e}}=(a^i(ax)^j)^{2^{e}}=
\begin{cases}
a^{2^{e}(i+j)}v^{2^{e-1}j}=v^{2^{e-1}j}
  & \text{if } \lambda\ge n-2,\\
a^{2^{e}(i+j)}v^{2^{e-1}u_2}=v^{2^{e-1}u_2}
  & \text{if } \lambda\le n-3.
\end{cases}\]
and
\[1\neq b^{2^{e-1}}=(a^i(ax)^j)^{2^{e-1}}=\begin{cases}
a^{2^{e-1}(i+j)}v^{2^{e-2}j}
  & \text{if } \lambda\ge n-2,\\
a^{2^{e-1}(i+j)}v^{2^{e-2}u_2}
  & \text{if } \lambda\le n-3.
\end{cases},\]
where $u_2$ is odd.
The first equation implies that $e\ge n+1$, and hence $f\le m\le e$.
The inequality in the second equation means that either $a^{2^{e-1}}\neq 1$ or $v^{2^{e-2}}\neq 1$; equivalently, $e\le m$ or $e\le n+1$.
Thus we have
\[(e,f)=(m,n+1)~~\text{or}~~ (e,f)=(n+1,m).\]
Assume that $(e,f)=(m,n+1)$. Then $m\ge n+1$ since $e\ge f$.
As $k$ is odd and by Eqs.~(4.1)--(4.2) we have
$$1\ne b^{2^{e-1}}=(a^i(ax)^j)^{2^{m-1}}
=a^{2^{m-1}}v^{2^{m-2}u_2}$$
and
$$1\neq w^{2^{f-1}}=(a^k(ax)^{\ell})^{2^n}=a^{2^n(k+\ell)}\in \l a\r.$$
Since $\langle b\rangle\cap\langle w\rangle=1$, it follows that
$\l b^{2^{e-1}}\r\cap\l w^{2^{f-1}}\r=1$, and so
$v^{2^{m-2}u_2}\ne1$, which implies that $m-2<n$, i.e., $m\le n+1$.
Thus, we have $m=n+1$ and $e=n+1=m=f$. In conclusion, $$(e,f)=(n+1,m), ~\mbox{and}~n+1\geq m.$$
In this case, Eqs.~(4.1)--(4.2) yield
\[1\neq b^{2^{e-1}}=(a^i(ax)^j)^{2^{n}}=a^{2^{n}}v^{2^{n-1}}.\]

If $\lambda\ge n-2$, then by Eq.~(4.1) we have
\[w^{2^{f-1}}=(a^k(ax)^\ell)^{2^{m-1}}=a^{2^{m-1}(k+\ell)}\neq 1,\]
which is disjoint from $\langle b^{2^{e-1}}\rangle$; hence $k+\ell$ is odd, i.e., $k$ is odd and $\ell$ is even.

If $\lambda\le n-3$, we can also deduce that $k+\ell$ is odd. Indeed,
suppose for contradiction that $k+\ell$ is even. Then by Eq.~(4.2),
\[1\neq w^{2^{f-1}}=a^{2^{m-1}(k+\ell)}v^{2^{m-1+\lambda+\nu_2(k+\ell)}u_1}
=v^{2^{m-1+\lambda+\nu_2(k+\ell)}u_1}\in \langle v\rangle.\]
Since $\langle b\rangle\cap\langle w\rangle=1$, we have $b^{2^{e-1}}=a^{2^{n}}v^{2^{n-1}}\notin \langle v\rangle$. This implies $a^{2^n}\neq 1$, hence $n\le m-1$.
Moreover, from the previous condition we must have $n+1=m$.
But in this case,
\[w^{2^{f-1}}=v^{2^{m-1+\lambda+\nu_2(k+\ell)}u_1}=1,\]a contradiction.
Therefore, $k+\ell$ is odd; that is, $k$ is odd and $\ell$ is even.
By Lemma~\ref{parameter-relation}, \(m+\lambda+1\ge n\), and consequently
\[w^{2^f}=(a^k(ax)^\ell)^{2^{m}}
=a^{2^{m}(k+\ell)}v^{2^{m+\lambda+\nu_2(\ell)}u_3}
=v^{2^{m+\lambda+\nu_2(\ell)}u_3}=1.\]

\vskip0.1in
\textbf{Case 1.2. Suppose now that $i$ is odd and \(j\) is even.} It follows
from Eq.~(4.6) that $\ell$ is odd, and by Lemma~\ref{exponent} and
Eqs.~(4.1)--(4.2) we have
\[1\ne b^{2^{e-1}}=(a^i(ax)^j)^{2^{e-1}}=a^{2^{e-1}(i+j)}\in \l a\r,\]
which means that $e=m$, and so $(e,f)=(m,n+1)$, where $m\geq n+1$ because $e\geq f$.  In this case,  $$b^{2^{e-1}}=(a^i(ax)^j)^{2^{m-1}}=a^{2^{m-1}}.$$
Suppose that $k$ is odd. Then by Eqs.~(4.1)--(4.2), we have
\[1\neq w^{2^{f-1}}=(a^k(ax)^\ell)^{2^{n}}=a^{2^{n}(k+\ell)}\in \l a\r,\]
which contradicts the fact that $\langle b\rangle\cap\langle w\rangle=1$. Thus, $k$ is even and $\ell$ is odd. Under this circumstance,
\[1= w^{2^{f}}=(a^k(ax)^\ell)^{2^{n+1}}=a^{2^{n+1}(k+\ell)}v^{2^{n}}=a^{2^{n+1}(k+\ell)},\]
which implies that $n+1\geq m$, and so $m=n+1$.
\vskip0.1in
From the discussion of the above two cases, we obtain
\[
(e,f)=(n+1,m),\qquad\text{where }n+1\ge m.
\]
Then, from Eq.~(4.4), it follows immediately that $o(ax)=2^{n+1}$ and so $1\le i,k\le 2^m$ and $1\le j,\ell\le 2^{n+1}$.
Consequently, the possible values of $i,j,k,\ell$ can be classified into the following categories.
\begin{itemize}
\item[(i)] If $n+1=m$, i.e., $e=f$, then either $i,\ell$ are both even and $j,k$ are both odd, or $i,\ell$ are both odd and $j,k$ are both even.
Thus\[|\Omega|=2\cdot(2^{m-1}\cdot 2^{m-1}\cdot 2^{n}\cdot 2^{n})=2^{4m-3},\]
as stated in part (1).
\item[(ii)] If $m-1<n\le m+\lambda+1$, then $i,\ell$ are both even and $j,k$ are both odd. That is, $i,j,k,\ell$ have $2^{m-1}$, $2^n$, $2^{m-1}$, and $2^n$ possible choices, respectively.
Hence\[|\Omega|=2^{m-1}\cdot 2^{m-1}\cdot 2^{n}\cdot 2^{n}=2^{2m+2n-2},\]
as stated in part (2).
\end{itemize}

\vskip0.1in
\textbf{Case 2:} \(\eta=1\), equivalently \(N\cap Z=\langle z\rangle\).
\vskip0.1in
Since $|Z/(Z\cap N)|=2^m$ with $m>1$, it follows that $Z=\l a\r\cong\ZZ_{2^{m+1}}$ and $z=a^{2^m}=v^{2^{n-1}}.$
Also, from $2^{e+f}=|G|=2^{m+n+1}$ and $\exp(G)=2^e$, we get
$$m+n+1=e+f,\quad\text{where }e\ge \max\{m+1,n\}\text{ and }f\le \min\{n,m+1\}.$$
Since $i$ and $j$ have different parity, we treat the two cases separately.

\vskip0.1in
\textbf{Case 2.1. Suppose first that $i$ is even and \(j\) is odd.} It follows
from Eq.~(4.6) that $k$ is odd, and by Eqs.~(4.1)--(4.2),
$$b^{2^e}=(a^i(ax)^j)^{2^e}=a^{2^e(i+j)}v^{2^{e-1}u_2}
=v^{2^{e-1}u_2}=1,$$
so $e\ge n+1$.

If $e>n+1$, then $e-2\geq n$ and it follows from Eqs.~(4.1)--(4.2) that
$$1\ne b^{2^{e-1}}=a^{2^{e-1}(i+j)}v^{2^{e-2}u_2}
=a^{2^{e-1}(i+j)},$$
hence $e-1\leq m$. Thus $e=m+1$ and $f=n$ with $m>n$. In this case, $b^{2^{e-1}}=a^{2^m}\in\langle a\rangle$.
On the other hand, since $k$ is odd and using Eqs.~(4.1)--(4.2), we have
$$1\neq w^{2^{f-1}}=(a^k(ax)^{\ell})^{2^{n-1}}=a^{2^{n-1}(k+\ell)}\in \l a\r,$$
so $\l w^{2^{f-1}}\r\cap \l b^{2^{e-1}} \r\neq 1$, contradicting $\langle w\rangle\cap\langle b\rangle=1$.
 Therefore, $$e=n+1~\mbox{and}~f=m.$$

Now $2^{n+1}=2^e=\exp(G)\ge o(a)=2^{m+1}$, so $n\ge m$. If $n=m$, then
$$b^{2^{e-1}}=(a^i(ax)^j)^{2^n}=a^{2^n(i+j)}v^{2^{n-1}}=a^{2^m}a^{2^m}=1,$$
a contradiction. Hence $n>m$, and we have
$$(e,f)=(n+1,m)\quad\text{with }n>m.$$
In particular, $b^{2^{e-1}}=(a^i(ax)^j)^{2^n}=v^{2^{n-1}}\in\langle v\rangle$.

Suppose that $\ell$ is odd. Then $k+\ell$ is even, and hence $$a^{2^{m-1}(k+\ell)}=1~~\mbox{or}~~a^{2^{m-1}(k+\ell)}=a^{2^m}=v^{2^{n-1}}.$$
Since $k$ is odd, Eqs.~(4.1)--(4.2) yield
\[1\ne w^{2^{f-1}}=(a^k(ax)^{\ell})^{2^{m-1}}=\begin{cases}
a^{2^{m-1}(k+\ell)} & \text{if } \lambda\ge n-2,\\
a^{2^{m-1}(k+\ell)}v^{2^{m-1+\lambda+\nu_2(k+\ell)}u_1}
  & \text{if } \lambda\le n-3.
\end{cases}\]
In either case,  $$1\ne w^{2^{f-1}}\in \l v\r,$$ contradicting $\langle b\rangle\cap\langle w\rangle=1$. Thus $\ell$ is even.

With $\ell$ even, we have
\[1=w^{2^f}=(a^k(ax)^\ell)^{2^m}=\begin{cases}
a^{2^m(k+\ell)}\ne1 & \text{if } \lambda\ge n-2,\\
a^{2^m(k+\ell)}v^{2^{m+\lambda+\nu_2(\ell)}u_3}
  =v^{2^{n-1}+2^{m+\lambda+\nu_2(\ell)}u}
  & \text{if } \lambda\le n-3.
\end{cases}\]
where \(u\) is odd. Hence \(\lambda\le n-3\) and
\[
2^{n-1}+2^{m+\lambda+\nu_2(\ell)}u\equiv0\pmod{2^n}.
\]
By Lemma~\ref{parameter-relation}, \(m+\lambda\ge n-1\), which means that
\(1+2^{m+\lambda+\nu_2(\ell)-(n-1)}u\equiv0\pmod2\), impossible.
Therefore Case 2.1 cannot occur.

\vskip0.1in
\textbf{Case 2.2. Suppose now that $i$ is odd and \(j\) is even.} It follows
from Eq.~(4.6) that $\ell$ is odd, and by Lemma~\ref{exponent}, we have
\(e\geq \max\{m+1,n\}\).
From Eqs.~(4.1)--(4.2), we obtain
\[
1\neq b^{2^{e-1}}=(a^i(ax)^j)^{2^{e-1}}=a^{2^{e-1}(i+j)},
\]
which implies \(e-1\leq m\), and hence \(e=m+1\) and \(f=n\).
Consequently, we have $b^{2^{e-1}}=a^{2^m}$.

Assume that $k$ is odd. Then $k+\ell$ is even and by
Eqs.~(4.1)--(4.2) we have
\[
1\neq w^{2^{f-1}}=(a^k(ax)^\ell)^{2^{n-1}}
=a^{2^{n-1}(k+\ell)}\in \l a\r,
\]
and so $\l b^{2^{e-1}}\r\cap \l w^{2^{f-1}}\r\neq 1$,
a contradiction. Thus, $k$ is even.

Since $$w^{2^f}=(a^k(ax)^\ell)^{2^{n}}=a^{2^n(k+\ell)}v^{2^{n-1}}=a^{2^n(k+\ell)}a^{2^m}=1,$$
it follows that $$2^n(k+\ell)+2^m\equiv 0~(\mbox{mod}~2^{m+1}),$$ and so $m=n.$ In this case, we have $$(e,f)=(n+1,m) ~\mbox{with}~m=n.$$
By Eq.~(4.4), $(ax)^{2^n}=x^{2^n}a^{2^n}v^{2^{n-1}u}=a^{2^m}v^{2^{n-1}}=1$, so $o(ax)=2^n$ and
\[\begin{array}{llll}
\Omega=\{(a^i(ax)^j,a^k(ax)^{\ell})&|&1\leq i,k\leq 2^{m+1}, 1\leq j,\ell\leq 2^{n},\\
                                   &~&j,k~\mbox{are both even},i,\ell~\mbox{are both odd}\},
\end{array} \] with cardinality $|\Omega|=2^{2n+2m-2}$,  as in part (3).
\qed

\vskip0.2in
\subsection{The automorphisms of \(G(\eta,t,\delta)\)}~~~~~~~~~~~~~~~~~~~~~~~
\vskip0.15in
Next, we determine the order of the automorphism group of
\(G=G(\eta,t,\delta)\).
Our calculation will use the first isomorphism theorem for the natural map
\(\Aut(G)\to\Aut(G/\langle v^2\rangle)\). The quotient \(G/\langle v^2\rangle\)
is isomorphic to a fixed group of order \(2^{m+2}\), and the image of
\(\Aut(G)\) embeds into the automorphism group of this quotient. We therefore
first determine the automorphism group of this model group.
Throughout this subsection, put \(\lambda=\nu_2(t)\), with the convention
\(\nu_2(0)=+\infty\).

\begin{lemma}\label{lem:aut-auxiliary}
Let $m>1$ and let
$H=\l x,y\mid x^{2^m}=y^2=z^2=1,\ [x,y]=z,\ [x,z]=[y,z]=1\r$.
Then $|\Aut(H)|=2^{m+3}$.
\end{lemma}
\proof Since \(\langle x,z\rangle=\langle x\rangle\times\langle z\rangle\) is
normal in \(H\), every element of \(H\) can be written uniquely as
\(h=x^iy^jz^k\), where \(1\le i\le2^m\) and \(j,k\in\{0,1\}\).
Then
\[
h^2=(x^iy^jz^k)^2=(x^iy^j)^2
=\begin{cases}
x^{2i} & \text{if } i \text{ is even, or } i \text{ is odd and } j=0,\\
x^{2i}z & \text{if } i \text{ is odd and } j=1.
\end{cases}
\]
For \(s\ge2\),
\[
h^{2^s}=x^{2^si}.
\]
Since \(m>1\), all elements of order \(2^m\) in \(H\) are of the form
\(x^iy^jz^k\) with \(1\le i\le2^m\) odd and \(j,k\in\{0,1\}\),
while all involutions in \(H\) are of the form
\(x^{2^{m-1}i_1}y^{j_1}z^{k_1}\), where
\(i_1,j_1,k_1\in\{0,1\}\) and \((i_1,j_1,k_1)\ne(0,0,0)\).

Since \(\langle z\rangle\) is a characteristic subgroup of \(H\), every
automorphism \(\sigma\in\Aut(H)\) has the form
\[\sigma:\left\{
\begin{array}{ll}
x \mapsto x^iy^jz^k, & 1\le i\le2^m\text{ odd},\quad j,k\in\{0,1\},\\
y \mapsto x^{2^{m-1}i_1}y^{j_1}z^{k_1},& i_1,j_1,k_1\in\{0,1\},\\
z \mapsto z.
\end{array}\right.\]

Because \(z=\sigma(z)=[\sigma(x),\sigma(y)]\), we compute
\[
[\sigma(x),\sigma(y)]
=[x^iy^jz^k,x^{2^{m-1}i_1}y^{j_1}z^{k_1}]
=[x^iy^j,y^{j_1}]\ne1.
\]
This forces \(j_1=1\). The remaining relations
\([\sigma(x),\sigma(z)]=[\sigma(y),\sigma(z)]=1\) are then automatically
satisfied.

Therefore, the standard form of $\sigma$ is
\[
\sigma:\left\{
\begin{array}{ll}
x \mapsto x^iy^jz^k, & 1\le i\le2^m\text{ odd},\quad j,k\in\{0,1\},\\
y \mapsto x^{2^{m-1}i_1}yz^{k_1}, & i_1,k_1\in\{0,1\},\\
z \mapsto z.
\end{array}\right.
\]
Distinct choices of \(i,j,k,i_1,k_1\) give distinct automorphisms. Hence
\[
|\Aut(H)|=2^{m-1}\cdot2^4=2^{m+3}.
\]
\qed

Now we determine the order of the automorphism group of $G=G(\eta,t,\delta)$.

\begin{theorem}\label{thm:aut-Geta}
Let \(G=G(\eta,t,\delta)\) be as in Lemma~\ref{lem:cyclic-commutator}(2).
Then
\begin{itemize}
\item[(1)] If \(\lambda\ge n-2\), then
\[
|\Aut(G)|=
\begin{cases}
2^{2n+m},&\eta=0,\\
2^{2n+m-1},&\eta=1.
\end{cases}
\]
\item[(2)] If \(\lambda\le n-3\), then
\(\lambda+3\le n\le m+\lambda+1\), and
\[
|\Aut(G)|=
\begin{cases}
2^{2n+1},&n=m+\lambda+1,\ \eta=0,\\
2^{2n+2},&n=m+\lambda+1,\ \eta=1,\\
2^{n+m+\lambda+3},&
  \lambda+3\le n\le m+\lambda,\ \eta=0,\\
2^{n+m+\lambda+2},&
  \lambda+3\le n\le m+\lambda,\ \eta=1.
\end{cases}
\]
\end{itemize}

\end{theorem}

\proof Let \(G=G(\eta,t,\delta)=NZ\), where \(m,n>1\), \(Z=\langle a\rangle\),
and
\[
N=\langle v,x\mid v^{2^n}=1,\ x^2\in\langle z\rangle,\
z=v^{2^{n-1}},\ v^x=v^{-1}\rangle\trianglelefteq G.
\]
Moreover, \(Z\cap N\le\langle z\rangle=\Z(N)\) and
\(|Z/(Z\cap N)|=2^m\). The generators satisfy
\[
[a,x]=v,\qquad [v,a^{2^{m-1}}]=1,\qquad v^a=v^{-1+4t}.
\]
Put \(r=-1+4t\).
Since $G'=\l v\r$ and $N'=\l v^2\r$,  it follows that $\l v^2\r$ is a characteristic subgroup of $G$.
Let
\[
\overline{G}=G/\langle v^2\rangle
=\langle\overline{v},\overline{x},\overline{a}\rangle,
\]
where \(\overline{v},\overline{x},\overline{a}\) denote the images of
\(v,x,a\), respectively.
Because $Z\cap N\leq \l z\r=\l v^{2^{n-1}} \r$ and $|Z/(Z\cap N)|=2^m$, we have $a^{2^m}\in \l z\r=\l v^{2^{n-1}}\r$. Thus
\[\overline{G}=\l \overline{v},\overline{x},\overline{a}~|~\overline{a}^{2^m}=\overline{v}^2=\overline{x}^2=1,
[\overline{a},\overline{x}]=\overline{v},[\overline{a},\overline{v}]=[\overline{x},\overline{v}]=1\r\cong H.\]

Notice that each \(\sigma\in\Aut(G)\) induces an automorphism
\(\overline{\sigma}\) of \(\overline{G}\) by
\[
(g\langle v^2\rangle)^{\overline{\sigma}}
=g^\sigma\langle v^2\rangle,
~\text{where }g\langle v^2\rangle\in\overline{G}.
\]
Thus there is a homomorphism
\[
\rho:\Aut(G)\longrightarrow\Aut(\overline{G})
\]
defined by \(\rho(\sigma)=\overline{\sigma}\). Let \(\Ker\rho\) and
\(\Aut(G)^\rho\) denote the kernel and image of \(\rho\), respectively.
The first isomorphism theorem gives
$\Aut(G)/\Ker\rho\cong\Aut(G)^\rho\le\Aut(\overline{G}).$
Consequently,
\[
|\Aut(G)|=|\Ker\rho|\,|\Aut(G)^\rho|.
\]
It therefore remains to determine \(|\Ker\rho|\) and
\(|\Aut(G)^\rho|\).

\vskip0.1in
\textbf{(1) Determining \(|\Ker\rho|\).}
\vskip0.1in
Explicitly,
\[
\Ker\rho=\{\sigma\in\Aut(G):g^\sigma\langle v^2\rangle
=g\langle v^2\rangle\text{ for all }g\in G\}.
\]
Since \(G=\langle v,x,a\rangle\), every nontrivial
\(\sigma\in\Ker\rho\) has the form
\[
\sigma:\left\{
\begin{array}{ll}
v\mapsto v^{2k+1},\\
x\mapsto xv^{2\ell},\\
a\mapsto av^{2s},
\end{array}
\right.
\qquad 1\le k,\ell,s\le2^{n-1}.
\]
We now determine all possible values of \(k,\ell,s\) in order to compute
\(|\Ker\rho|\).

Since \(\sigma\in\Aut(G)\), it preserves the orders of the
generators \(v,x,a\) and leaves the defining relations invariant:
\[
(v^x)^\sigma=(v^{-1})^\sigma,\quad
[a,x]^\sigma=v^\sigma,\quad
(v^a)^\sigma=(v^{-1+4t})^\sigma,\quad
[v,a^{2^{m-1}}]^\sigma=1.
\]
We begin with the relations involving \(v\) and \(x\). Since
\(v^\sigma=v^{2k+1}\) and
\[
(x^\sigma)^2=(xv^{2\ell})^2
=x^2(v^{2\ell})^xv^{2\ell}=x^2,
\]
we have \(|v^\sigma|=|v|\), \(|x^\sigma|=|x|\), and
\[
(v^x)^\sigma=(v^{2k+1})^{xv^{2\ell}}
=v^{-(2k+1)}=(v^{-1})^\sigma.
\]
Similarly,
\[
(v^a)^\sigma=(v^{2k+1})^{av^{2s}}=v^{r(2k+1)}=(v^r)^\sigma.
\]
Since \([v,a^{2^{m-1}}]=1\), equation (4.3) also gives
\[
[v,a^{2^{m-1}}]^\sigma
=[v^{2k+1},(av^{2s})^{2^{m-1}}]
=[v^{2k+1},a^{2^{m-1}}
v^{2s(r+1)(r^2+1)\cdots(r^{2^{m-2}}+1)}]=1.
\]
It remains to impose \(|a^\sigma|=|a|\) and
\([a,x]^\sigma=v^\sigma\). By (4.3),
\[
(a^\sigma)^{2^m}=(av^{2s})^{2^m}
=\begin{cases}
a^{2^m},&\lambda\ge n-2,\\
a^{2^m}v^{2s(r+1)(r^2+1)\cdots(r^{2^{m-1}}+1)},&
  \lambda\le n-3.
\end{cases}
\]
In the first case, \(|a^\sigma|=|a|\) follows immediately. In the second,
\(\lambda=\nu_2(t)\); write \(t=2^\lambda t_0\) with \(t_0\) odd. Since
\(2\parallel(r^{2^d}+1)\) for every \(d\ge1\) and
\(2^{2+\lambda}\parallel(r+1)\), we obtain
\[
2^{m+\lambda+1}\parallel
(r+1)(r^2+1)\cdots(r^{2^{m-1}}+1).
\]
By Lemma~\ref{parameter-relation}, \(m+\lambda+1\ge n\), so the displayed
divisibility implies that the \(v\)-factor is \(1\). Consequently,
\((a^\sigma)^{2^m}=a^{2^m}\), and hence \(|a^\sigma|=|a|\) for every
\(1\le s\le2^{n-1}\).

Finally, the relation \([a,x]^\sigma=v^\sigma\) gives
\[
v^{2k+1}=v^\sigma=[a,x]^\sigma=[av^{2s},xv^{2\ell}]
=v^{1-4s+4\ell(1-2t)}.
\]
Consequently,
\[
k\equiv-2s+2\ell(1-2t)\pmod{2^{n-1}}.
\]
Since \(1\le k\le2^{n-1}\) is a complete residue system modulo
\(2^{n-1}\), each pair \((s,\ell)\) determines exactly one admissible
\(k\). Therefore, since \(s\) and \(\ell\) each have \(2^{n-1}\) choices,
\[
|\Ker\rho|=2^{2n-2}. \eqno(4.7)
\]

\vskip0.1in
\textbf{(2) Determining \(|\Aut(G)^\rho|\).}
\vskip0.1in
Recall that \(\Aut(G)^\rho\le\Aut(\overline{G})\) and
\(\overline{G}\cong H\). Hence, by Lemma~\ref{lem:aut-auxiliary},
\[
|\Aut(\overline{G})|=|\Aut(H)|=2^{m+3}.
\]
To determine \(|\Aut(G)^\rho|\), it is enough to compute the index
\(|\Aut(\overline{G}):\Aut(G)^\rho|\).

Let \(\overline{\tau}\in\Aut(\overline{G})\) be arbitrary. By the proof of
Lemma~\ref{lem:aut-auxiliary}, $\overline{\tau}$ has the form
\[
\overline{\tau}:\left\{
\begin{array}{ll}
\overline{a}\mapsto\overline{a}^i\overline{x}^j\overline{v}^k,
  & 1\le i\le2^m\text{ odd},\quad j,k\in\{0,1\},\\
\overline{x}\mapsto
  \overline{a}^{2^{m-1}i_1}\overline{x}\overline{v}^{k_1},
  & i_1,k_1\in\{0,1\},\\
\overline{v}\mapsto\overline{v}.
\end{array}\right.
\]
We determine which choices admit a lift \(\tau\in\Aut(G)\) satisfying
\(\rho(\tau)=\overline{\tau}\). Let \(\tau\) be such a lift. Since \(G\) is
generated by \(a\) and \(x\), the lift is determined by the images of \(a\)
and \(x\); its image of \(v=[a,x]\) is then determined as well. Hence
\(\tau\) has the form
\[\tau:\left\{
\begin{array}{ll}
a \mapsto a^ix^jv^{k+2\ell},
  & 1\le i\le2^m\text{ odd},\quad j,k\in\{0,1\},\\
x \mapsto a^{2^{m-1}i_1}xv^{k_1+2\ell_1},
  & i_1,k_1\in\{0,1\},\\
v \mapsto v^{1+2\ell_2},
  & 1\le\ell,\ell_1,\ell_2\le2^{n-1}.
\end{array}\right.\]
The parameters \(\ell,\ell_1,\ell_2\) encode the additional freedom inside
the same kernel coset and therefore do not affect whether \(\overline{\tau}\)
belongs to \(\Aut(G)^\rho\). Thus the index will be determined only by
\((i,j,k,i_1,k_1)\).
Since $\tau\in \Aut(G)$, it preserves the orders of $a,x,v$ and leaves the defining relations invariant:
\[
(v^x)^\tau=(v^{-1})^\tau,\quad
[a,x]^\tau=v^\tau,\quad
(v^a)^\tau=(v^{-1+4t})^\tau,\quad
[v,a^{2^{m-1}}]^\tau=1.
\]

The conditions $|v^\tau| = |v|$ and $(v^x)^\tau = (v^{-1})^\tau$ hold automatically. We now examine the remaining constraints, beginning with $|x^\tau| = |x|$.
We have \(x^\tau=a^{2^{m-1}i_1}xv^{k_1+2\ell_1}\), where
\(i_1\in\{0,1\}\).
Squaring gives
\[
(x^\tau)^2=(a^{2^{m-1}i_1}x)^2(v^{k_1+2\ell_1})^xv^{k_1+2\ell_1}
=x^2a^{2^{m-1}i_1}(a^x)^{2^{m-1}i_1}
=x^2a^{2^{m-1}i_1}(av)^{2^{m-1}i_1}.
\]
By Eq.~(4.3), we obtain
\[
(x^\tau)^2=x^2a^{2^mi_1}
v^{(r^{2^{m-1}i_1}-1)/(r-1)},~\mbox{where}~ i_1\in\{0,1\}.
\]
If \(\lambda\le n-3\) and \(i_1=1\), then
\(2^{m+\lambda+1}\parallel(r^{2^{m-1}}-1)\), and consequently
\[
\nu_2\!\left(\frac{r^{2^{m-1}}-1}{r-1}\right)=m+\lambda.
\]
Thus we may write
$\frac{r^{2^{m-1}}-1}{r-1}=2^{m+\lambda}u$ with $u$ odd.
Therefore \[(x^\tau)^2 =x^2a^{2^mi_1}v^{\frac{r^{2^{m-1}i_1}-1}{r-1}}=\begin{cases}
x^2a^{2^mi_1} & \text{if } \lambda\ge n-2,\\
x^2 & \text{if } \lambda\le n-3\text{ and }i_1=0,\\
x^2a^{2^m}v^{2^{m+\lambda}u}
  & \text{if } \lambda\le n-3\text{ and }i_1=1.
\end{cases}\]
If $\lambda\ge n-2$, then $|x^\tau|=|x|$ holds if and only if
\[
i_1=0\quad\text{or}\quad(i_1=1\text{ and }\eta=0).
\]
If $\lambda\le n-3$, then $|x^\tau|=|x|$ holds if and only if
\[
i_1=0\quad\text{or}\quad(i_1=1\text{ and }
\begin{cases}
		v^{2^{m+\lambda}}=1 & \text{if } \eta=0,\\
		v^{2^{n-1}+2^{m+\lambda}u}=1 & \text{if } \eta=1.
	\end{cases}).
\]
Equivalently,
\[
i_1=0\quad\text{or}\quad(i_1=1\text{ and }
\begin{cases}
n\le m+\lambda & \text{if } \eta=0,\\
n=m+\lambda+1 & \text{if } \eta=1.
	\end{cases}).
\]
Consequently, \(i_1=1\) is possible if and only if one of the following
conditions holds:
\begin{itemize}
\item \(\eta=0\) and \(\lambda\ge n-2\);
\item \(\eta=0\) and \(\lambda+3\le n\le m+\lambda\);
\item \(\eta=1\) and \(n=m+\lambda+1\).
\end{itemize}
Otherwise, only $i_1=0$ is allowed.

Now we verify the equation \((v^a)^\tau=(v^r)^\tau\).
Since $(v^r)^\tau=(v^\tau)^r=v^{(2\ell_2+1)r}$ and
\[
(v^a)^\tau=(v^\tau)^{a^\tau}
=(v^{2\ell_2+1})^{a^ix^jv^{k+2\ell}}
=(v^{(2\ell_2+1)r^i})^{x^j},
~\mbox{where}~ j\in\{0,1\}.
\]
Suppose first that \(j=1\). Then
\[
v^{(2\ell_2+1)r}=(v^a)^\tau
=(v^{(2\ell_2+1)r^i})^x
=v^{-(2\ell_2+1)r^i},
\]
which implies
\[
r(2\ell_2+1)(1+r^{i-1})\equiv0\pmod{2^n}.
\]
Since \(r(2\ell_2+1)\) is odd, it follows that
\(2^n\mid(1+r^{i-1})\). However, this is impossible because \(i\) is odd
and \(n>1\). Hence \(j=0\), and consequently
\[
a^\tau=a^iv^{k+2\ell}.
\]
With \(j=0\), the original equation becomes
\[
v^{(2\ell_2+1)r}=(v^r)^\tau=(v^a)^\tau=v^{(2\ell_2+1)r^i},
\]
which yields
\[
r(2\ell_2+1)(r^{i-1}-1)\equiv0\pmod{2^n}.
\]
Again, since \(r(2\ell_2+1)\) is odd, we have
\[
r^{i-1}-1\equiv0\pmod{2^n}.
\]
If \(\lambda\ge n-2\), that is,
\(r=-1+4t\equiv-1\pmod{2^n}\), then \(i\) can
be any odd integer in \([1,2^m]\); hence there are \(2^{m-1}\) choices for
\(i\). If $\lambda\le n-3$, then from
$\nu_2(r^{i-1}-1)=\lambda+2+\nu_2(i-1)$ we obtain
\[\nu_2(i-1)\geq n-\lambda-2.\]
Thus the number of choices for the odd integer $i$ is $2^{m-(n-\lambda-2)}$.
In conclusion,
\begin{itemize}
\item If \(\lambda\ge n-2\), there are \(2^{m-1}\) choices for \(i\);
\item If \(\lambda\le n-3\), there are \(2^{m-(n-\lambda-2)}\) choices for
\(i\).
\end{itemize}

Moreover, a direct calculation gives
\[
(a^\tau)^{2^{m-1}}=(a^iv^{k+2\ell})^{2^{m-1}}
=a^{2^{m-1}i}
v^{(k+2\ell)(r^i+1)(r^{2i}+1)\cdots(r^{2^{m-2}i}+1)}.
\]
Since \([v,a^{2^{m-1}}]=1\), it follows that
\[
[v,a^{2^{m-1}}]^\tau
=[v^\tau,(a^\tau)^{2^{m-1}}]
=[v^{1+2\ell_2},(a^\tau)^{2^{m-1}}]=1.
\]

We now verify that \(|a^\tau|=|a|\). Since
\(a^\tau=a^iv^{k+2\ell}\), we have
\[
(a^\tau)^{2^m}=(a^iv^{k+2\ell})^{2^m}
=a^{2^mi}
v^{(k+2\ell)(r^i+1)(r^{2i}+1)\cdots(r^{2^{m-1}i}+1)}.
\]
Recall that \(i\) is odd. If \(\lambda\ge n-2\), then
\((a^\tau)^{2^m}=a^{2^mi}=a^{2^m}\). If \(\lambda\le n-3\), then
\(\nu_2(r^i+1)=2+\lambda\) and \(\nu_2(r^{2^s i}+1)=1\) with \(s\ge1\). Hence
$v^{(r^i+1)(r^{2i}+1)\cdots(r^{2^{m-1}i}+1)}=v^{2^{m+\lambda+1}u}$ for some odd integer $u$.
By the relation between $n$ and $\lambda$ established in Lemma~\ref{parameter-relation}(namely $n \le m+\lambda+1$), we have
\[v^{(r^i+1)(r^{2i}+1)\cdots(r^{2^{m-1}i}+1)}= v^{2^{m+\lambda+1}u} = 1.\]
Therefore
\[(a^\tau)^{2^m} = a^{2^m i} = a^{2^m}.\]
In summary, the condition $|a^\tau| = |a|$ imposes no further restriction on $k$; in other words, $k$ remains arbitrary subject to the constraints already derived.

Finally, we examine the equation \(v^\tau=[a,x]^\tau\). Since
\[
v^{2\ell_2+1}=v^\tau=[a,x]^\tau
=[a^iv^{k+2\ell},a^{2^{m-1}i_1}xv^{k_1+2\ell_1}]
=v^{-2(k+2\ell)+\frac{r^i-1}{r-1}[1+2(1-2t)(k_1+2\ell_1)]},
\]
we obtain
\[
2\ell_2+1\equiv -2(k+2\ell)
+\frac{r^i-1}{r-1}[1+2(1-2t)(k_1+2\ell_1)]\pmod{2^n}.
\]
Put \(w=(r^i-1)/(r-1)\). Then
\[
(w-1)+2w(1-2t)(k_1+2\ell_1)-2(k+2\ell)-2\ell_2
\equiv0\pmod{2^n}.
\]
Since \(w-1=r+\cdots+r^{i-1}\equiv0\pmod4\) and \(n>1\), reduction modulo
\(2\) gives
\[
w(1-2t)(k_1+2\ell_1)-(k+2\ell)-\ell_2\equiv0\pmod2.
\]
Since \(w(1-2t)\) is odd, this is equivalent to
\[
k_1-k\equiv\ell_2\pmod2.
\]
As \(\ell_2\) can have either parity, all four pairs
\((k,k_1)\in\{0,1\}^2\) are possible.

This completes the verification of the conditions for a lift \(\tau\). We now
count the resulting automorphisms in \(\Aut(\overline{G})\), considering the
two cases \(\lambda\ge n-2\) and \(\lambda\le n-3\) separately.
\begin{itemize}
\item[(1)] If \(\lambda\ge n-2\), then \(i\) is an arbitrary odd integer in \([1,2^m]\), and
either \(i_1=0\), or \(i_1=1\) and \(\eta=0\). The lift always forces
\(j=0\), which accounts for a factor \(2\) in the index. If \(\eta=0\), then
\(i_1=1\) remains possible. If \(\eta=1\), both \(j=0\) and \(i_1=0\) are
forced, contributing an additional factor \(2\). Therefore,
\[
|\Aut(\overline{G}):\Aut(G)^\rho|=
\begin{cases}
2,&\eta=0,\\
2^2,&\eta=1.
\end{cases}
\]
Using \(|\Aut(\overline{G})|=2^{m+3}\) and
\(|\Ker\rho|=2^{2n-2}\), we obtain
\[
|\Aut(G)|=
\begin{cases}
2^{2n+m},&\eta=0,\\
2^{2n+m-1},&\eta=1.
\end{cases}
\]

\item[(2)] If \(\lambda\le n-3\), then by Lemma~\ref{parameter-relation},
\(\lambda+3\le n\le m+\lambda+1\), and \(i\) has
\(2^{m-n+\lambda+2}\) choices. The criterion for \(i_1=1\) above separates
\(n\le m+\lambda\) from \(n=m+\lambda+1\). Accordingly, we distinguish four
cases according to the value of \(\eta\) and whether \(n\) attains the upper
endpoint.
In all four cases below, \(j\) is forced to be \(0\), contributing a common
factor \(2\) to the index, whereas \(k,k_1\in\{0,1\}\) remain arbitrary and
therefore cause no loss. It remains only to track \(i_1\) and \(i\).
\begin{itemize}
\item[(i)] Suppose that \(\eta=0\) and
\(\lambda+3\le n\le m+\lambda\). Then \(i_1\in\{0,1\}\), so \(i_1\) causes
no loss. The parameter \(i\) originally has \(2^{m-1}\) possible values, but
now has only \(2^{m-n+\lambda+2}\). Thus \(i\) contributes the factor
\[
\frac{2^{m-1}}{2^{m-n+\lambda+2}}
=2^{n-\lambda-3}.
\]
to the index. Including the common loss from \(j\), we obtain
\[
|\Aut(\overline{G}):\Aut(G)^\rho|
=2^{n-\lambda-3}\cdot2
=2^{n-\lambda-2}.
\]
Since \(|\Aut(\overline{G})|=2^{m+3}\), we have
$|\Aut(G)^\rho|=2^{m-n+\lambda+5}.$
Using \(|\Ker\rho|=2^{2n-2}\), we get
\[
|\Aut(G)|=2^{m+n+\lambda+3}.
\]

\item[(ii)] Suppose that \(\eta=0\) and \(n=m+\lambda+1\). Then
the preceding criterion for \(i_1=1\) gives \(i_1=0\). The parameter \(i\)
originally has \(2^{m-1}\) choices but now has only
\(2^{m-n+\lambda+2}=2\), contributing a factor \(2^{m-2}\). Together with
the common loss from \(j\) and the loss from \(i_1\), we have
\[
|\Aut(\overline{G}):\Aut(G)^\rho|
=2^{m-2}\cdot2\cdot2=2^m.
\]
Thus \(|\Aut(G)^\rho|=2^3\), and hence
\[
|\Aut(G)|=2^{2n-2}\cdot2^3=2^{2n+1}.
\]

\item[(iii)] Suppose that \(\eta=1\) and \(n=m+\lambda+1\). Then \(i_1\)
still belongs to \(\{0,1\}\), so it causes no loss, while \(i\) again has
\(2\) choices and contributes a factor \(2^{m-2}\). Together with the common
loss from \(j\), we have
\[
|\Aut(\overline{G}):\Aut(G)^\rho|
=2^{m-2}\cdot2=2^{m-1},
\]
and therefore $
|\Aut(G)^\rho|=2^4.$
Consequently,
\[
|\Aut(G)|=2^{2n-2}\cdot2^4=2^{2n+2}.
\]

\item[(iv)] Suppose that \(\eta=1\) and
\(\lambda+3\le n\le m+\lambda\). Then \(i_1=0\), and \(i\) has
\(2^{m-n+\lambda+2}\) choices, contributing the factor
\(2^{n-\lambda-3}\). Together with the common loss from \(j\) and the loss
from \(i_1\), we obtain
\[
|\Aut(\overline{G}):\Aut(G)^\rho|
=2^{n-\lambda-3}\cdot2^2
=2^{n-\lambda-1},
\]
Since \(|\Aut(\overline{G})|=2^{m+3}\),
\[
|\Aut(G)^\rho|=2^{m-n+\lambda+4}.
\]
Using \(|\Ker\rho|=2^{2n-2}\), we obtain
\[
|\Aut(G)|=2^{m+n+\lambda+2}.
\]
\end{itemize}
\end{itemize}
This completes the proof.\qed

\vskip0.2in
\subsection{Classification and enumeration of regular dessins}~~~~~~~~~~~~~~~~~
\vskip0.1in
In this subsection we classify and enumerate the regular dessins with
underlying graph \(\K_{2^e,2^f}\) whose automorphism group belongs to the
family \(G(\eta,t,\delta)\). By Theorem~\ref{thm:group-classification}, the
isomorphism classes in this family are known. Lemma~\ref{lem:generating-pairs-second}
then determines which of these classes admit exact bicyclic generating pairs;
groups without such pairs do not contribute to the enumeration.

For every remaining group \(G\), let \(\Omega_G\) denote the set of exact
bicyclic generating pairs of \(G\) of type \((2^e,2^f)\).
Lemma~\ref{semi-regular} shows that the corresponding regular dessins
correspond bijectively to the
\(\Aut(G)\)-orbits on the set \(\Omega_G\), and the number of orbits is
\(|\Omega_G|/|\Aut(G)|\). By Lemma~\ref{lem:generating-pairs-second},
\(|\Omega_G|\) does not depend on the representative \(G\); denote this
common size by \(|\Omega|\). Hence, for a set \(\mathcal R\) of representatives,
the number of non-isomorphic regular dessins is
\[
\sum_{G\in\mathcal R}\frac{|\Omega_G|}{|\Aut(G)|}
=|\Omega|\sum_{G\in\mathcal R}\frac{1}{|\Aut(G)|}.
\]

\begin{lemma}\label{lem:count-second}
Let $n(\mathcal{K}_{2^e,2^f})$ denote the number of non-isomorphic regular dessins of $\K_{2^e,2^f}$ associated with the family $G(\eta,t,\delta)$. Then
\[n(\mathcal{K}_{2^e,2^f})=\begin{cases}
    2^{2f-3} & \text{if } e=f\ge3,\\
    2^{2f-2} & \text{if } e>f.
\end{cases}\]
\end{lemma}

\proof
Let \(G=G(\eta,t,\delta)\), and write \(\lambda=\nu_2(t)\), with the same
convention \(\nu_2(0)=+\infty\) as before. Then
Lemma~\ref{lem:generating-pairs-second} gives
\[
(e,f)=(n+1,m),
\quad\text{where}\quad
m-1\le n\le m+\lambda+1.
\]
More precisely, \(G\) admits exact bicyclic generating pairs precisely in
the following cases:
\begin{itemize}
\item \(\eta=0\), \(n=m-1\), and
  \(|\Omega|=2^{4m-3}\);
\item \(\eta=0\), \(m\le n\le m+\lambda+1\), and
  \(|\Omega|=2^{2n+2m-2}\);
\item \(\eta=1\), \(n=m\), and
  \(|\Omega|=2^{2n+2m-2}\).
\end{itemize}

By Theorem~\ref{thm:group-classification}, the isomorphism classes are
\[ G(0,2^{\lambda_0},0),~G(1,2^{\lambda_0},0),~G(\eta,2^{\lambda_0},1), \]
where $\max\{ 0,n-m-1 \}\le \lambda_0\le n-2$ and either $\eta=0$ if $\lambda_0\ge n-m$, or $\eta=1$ if $\lambda_0=n-m-1$.
We now determine which of these classes admit exact bicyclic generating
pairs. The endpoint class \(G(1,2^{\lambda_0},1)\) cannot occur: if it
admitted an exact bicyclic generating pair,
Lemma~\ref{lem:generating-pairs-second} would give \(m=n\), whereas
Theorem~\ref{thm:group-classification} places this class at
\(\lambda_0=n-m-1=-1<0\), impossible. Hence, within
\(\lambda_0\in\mathcal L_{m,n}\), the contributing classes are
\[
G(0,2^{\lambda_0},0),\qquad
G(1,2^{\lambda_0},0)\ (n=m),\qquad
G(0,2^{\lambda_0},1)\ (\lambda_0\ge n-m).
\]
We therefore distinguish the cases \(n=m-1\), \(n=m\), and \(n>m\) below.

\vskip0.1in
{\bf Case 1: \(G=G(0,2^{\lambda_0},0)\).}
\vskip0.1in
By Lemma~\ref{lem:generating-pairs-second},
\[ |\Omega|=\begin{cases}
2^{4m-3} & \text{if } n=m-1,\\
2^{2n+2m-2} & \text{if } m\le n\le m+\lambda_0+1,
\end{cases} \]
and by Theorem~\ref{thm:aut-Geta},
\[ |\Aut(G(0,2^{\lambda_0},0))|=\begin{cases}
2^{2n+m} & \text{if } \lambda_0=n-2,\\
2^{n+m+\lambda_0+3} & \text{if } n-m\le\lambda_0\le n-3,\\
2^{2n+1} & \text{if } \lambda_0=n-m-1.
\end{cases} \]
We now compute the contribution of this class according to the possible
values of \(n\).
\begin{itemize}
\item If \(n=m-1\), then
\(\max\{0,n-m-1\}=0\) and hence
\(0\le\lambda_0\le n-2\). Then
\[
\sum_{\lambda_0=0}^{n-2}\frac{|\Omega|}{|\Aut(G)|}
=\sum_{\lambda_0=0}^{n-3}\frac{2^{4m-3}}{2^{n+m+\lambda_0+3}}
+\frac{2^{4m-3}}{2^{2n+m}}
=\sum_{\lambda_0=0}^{m-4}2^{2m-\lambda_0-5}+2^{m-1}
=2^{2m-4}.
\]
\item If \(m=n\), then \(\max\{0,n-m-1\}=0\), and hence
\(0\le\lambda_0\le n-2\). Then
\[
\sum_{\lambda_0=0}^{n-2}\frac{|\Omega|}{|\Aut(G)|}
=\sum_{\lambda_0=0}^{n-3}\frac{2^{4m-2}}{2^{2m+\lambda_0+3}}
+\frac{2^{4m-2}}{2^{3m}}
=\sum_{\lambda_0=0}^{m-3}2^{2m-\lambda_0-5}+2^{m-2}
=2^{2m-4}.
\]
\item If \(m<n\), then \(n-m-1\ge0\), and hence
\(n-m-1\le\lambda_0\le n-2\). Then
\[
\begin{aligned}
\sum_{\lambda_0=n-m-1}^{n-2}\frac{|\Omega|}{|\Aut(G)|}
&=\frac{2^{2n+2m-2}}{2^{2n+1}}
+\sum_{\lambda_0=n-m}^{n-3}
  \frac{2^{2n+2m-2}}{2^{n+m+\lambda_0+3}}
+\frac{2^{2n+2m-2}}{2^{2n+m}}\\
&=2^{2m-3}
+\sum_{\lambda_0=n-m}^{n-3}2^{n+m-\lambda_0-5}
+2^{m-2}=3\cdot2^{2m-4}.
\end{aligned}
\]
\end{itemize}

\vskip0.1in
{\bf Case 2: \(G=G(0,2^{\lambda_0},1)\).}
\vskip0.1in
By Theorem~\ref{thm:group-classification}, this representative occurs
precisely when \(\lambda_0\ge n-m\). Moreover, \(|\Omega|\) and
\(|\Aut(G)|\) have the same formulas as in Case~1. We again distinguish
three possibilities.
\begin{itemize}
\item If \(n=m-1\), then
\[
0=\max\{0,n-m\}\le\lambda_0\le n-2.
\]
This is the same range as in Case~1; moreover, \(|\Omega|\) and
\(|\Aut(G)|\) take the same values. Hence the contribution is also
\(2^{2m-4}\).
\item If \(n=m\), then
\[
0=\max\{0,n-m\}\le\lambda_0\le n-2,
\]
which is again the same range and the same parameter formulas as in
Case~1. Therefore, the contribution is \(2^{2m-4}\).
\item If \(m<n\), then
\[
n-m=\max\{0,n-m\}\le\lambda_0\le n-2.
\]
Therefore,
\[
\sum_{\lambda_0=n-m}^{n-2}\frac{|\Omega|}{|\Aut(G)|}
=\sum_{\lambda_0=n-m}^{n-3}
\frac{2^{2n+2m-2}}{2^{n+m+\lambda_0+3}}
+\frac{2^{2n+2m-2}}{2^{2n+m}}
=\sum_{\lambda_0=n-m}^{n-3}2^{n+m-\lambda_0-5}+2^{m-2}
=2^{2m-4}.
\]
\end{itemize}

\vskip0.1in
{\bf Case 3: \(G=G(1,2^{\lambda_0},0)\).}
\vskip0.1in
By Lemma~\ref{lem:generating-pairs-second}, this class occurs only when
\(m=n\). In this case
\[
|\Omega|=2^{4m-2},
\]
and Theorem~\ref{thm:group-classification} gives
\(0\le\lambda_0\le m-2\).
Theorem~\ref{thm:aut-Geta} gives
\[ |\Aut(G(1,2^{\lambda_0},0))|=\begin{cases}
2^{3m-1} & \text{if } \lambda_0=m-2,\\
2^{2m+\lambda_0+2} & \text{if } 0\le\lambda_0\le m-3.
\end{cases} \]
Therefore,
\[
\sum_{\lambda_0=0}^{m-2}\frac{|\Omega|}{|\Aut(G)|}
=\sum_{\lambda_0=0}^{m-3}\frac{2^{4m-2}}{2^{2m+\lambda_0+2}}
+\frac{2^{4m-2}}{2^{3m-1}}
=\sum_{\lambda_0=0}^{m-3}2^{2m-\lambda_0-4}+2^{m-1}
=2^{2m-3}.
\]

Adding the contributions from the three isomorphism classes gives
\[
\sum_{G\in\mathcal R}\frac{|\Omega_G|}{|\Aut(G)|}
=
\begin{cases}
2\cdot2^{2m-4}=2^{2m-3},&n=m-1,\\
2^{2m-4}+2\cdot2^{2m-3}+2^{2m-4}=2^{2m-2},&n=m,\\
3\cdot2^{2m-4}+2^{2m-4}=2^{2m-2},&n>m.
\end{cases}
\]
Thus the last two cases have the same total, and hence they combine into
\(2^{2m-2}\) for \(n\ge m\).
Since \((e,f)=(n+1,m)\), the case \(n=m-1\) is precisely \(e=f\),
while \(n\ge m\) is precisely \(e>f\). Therefore,
\[ n(\mathcal{K}_{2^e,2^f})=\begin{cases}
2^{2f-3} & \text{if } e=f,\\
2^{2f-2} & \text{if } e>f.
\end{cases} \]
This completes the proof.\qed

{\noindent\bf Proof of Theorem~\ref{main-1}:}
The first part follows from Lemma~\ref{lem:G0-generating-pairs}, which shows that the first family can occur
only when $n=2$; the resulting group is $G_0$. The second part follows from
Theorem~\ref{thm:group-classification}, which determines the isomorphism classes of the groups
$G(\eta,t,\delta)$, and from Lemma~\ref{lem:generating-pairs-second}, which determines which of these
groups admit exact bicyclic factorizations and gives $(e,f)=(n+1,m)$.
\vskip0.2in

{\noindent\bf Proof of Theorem~\ref{main-2}:}
For $G=G_0$, the assertion follows from Lemma~\ref{lem:count-G0}.
For $G=G(\eta,t,\delta)$, the possibility $e=f=2$ does not occur because it would require
$(e,f)=(n+1,m)$ with $m,n>1$. If $e=f\geq3$ or $e>f$, the assertion follows from
Lemma~\ref{lem:count-second}.

\vskip0.2in

\noindent{\bf Data availability}
\vskip0.1in
No data was used for the research described in the article.

\vskip0.2in

\noindent{\bf Declaration of competing interest}
\vskip0.1in
We have no conflict of interest.

\end{document}